\RequirePackage{amsmath}
\documentclass[runningheads]{llncs}
\usepackage[T1]{fontenc}
\usepackage{graphicx}
\usepackage{amssymb}
\usepackage[colorlinks=true, allcolors=blue]{hyperref}
\usepackage[numbers]{natbib}
\usepackage{cleveref}
\usepackage[dvipsnames]{xcolor}
\usepackage{comment}

\newtheorem{open}{Open Question}

\newcommand{\ceil}[1]{\left\lceil #1 \right\rceil}
\newcommand{\floor}[1]{\left\lfloor #1 \right\rfloor}

\newcommand{\propcond}{\textsc{Bal}}
\newcommand{\weakpropcond}{\textsc{WeakBal}}

\newcommand{\range}[3]{\{#1_i: #2\leq i\leq #3\}}
\newcommand{\rangeA}[2]{\range{A}{#1}{#2}}
\newcommand{\rangeB}[2]{\range{B}{#1}{#2}}
\newcommand{\rangeC}[2]{\range{C}{#1}{#2}}
\newcommand{\rangeD}[2]{\range{D}{#1}{#2}}

\newif\ifRowIsPlayer
\RowIsPlayerfalse

\begin{document}
\title{Perpetually Fair Assignments Via Balanced Sequences of Permutations}
%
%

\author{Terrence Adams\inst{1} \and
Erel Segal-Halevi\inst{2}}
\authorrunning{T. Adams and E. Segal-Halevi}
%
\institute{University at Albany, State University of New York,
1400 Washington Avenue, Albany, NY 12222.
\email{terry@ieee.org}
\and
Ariel University, Ariel 40700, Israel.
\email{erelsgl@gmail.com}}

\maketitle              
\begin{abstract}
There is a set of $n$ indivisible items (goods or chores), and a set of $n$ players. 
Each day, a single item should be assigned to each player.
Assignments based on latin squares guarantee fairness after every $n$ days; our goal is to ensure fairness after every single day.

We present two 'balance' conditions on latin squares.
Informally, a latin square is balanced if its top rows and leftmost columns contain all $n$ labels; this ensures that all $n$ players receive one of the top items in one of the early days.
One such condition can always be satisfied, but is arguably too weak; a second condition is strong,  and can be satisfied for all $n\leq 12$, but cannot be satisfied for some larger values of $n$, including all $n>108$.

We show that the second balance condition guarantees that the cumulative assignment is always \emph{proportional up to one item (PROP1)}, where proportionality holds in a strong ordinal sense --- for every valuations that are consistent with the item ranking. 

Finally, we present a weaker balance condition on a sequence, that guarantees ordinal proportionality up to two items (PROP2). Whether or not this condition can be satisfied for all $n$ remains an open question.
\\
~\\
\emph{This is the full version of a paper accepted to SAGT 2026 conference.}

\end{abstract}
\section{Introduction}
Some $n$ house-mates need to divide the $n$ daily house-chores among them. Each housemate must do exactly one chore per day. What would be a fair way to assign the chores?
Ideally, a fair allocation would give each housemate the same amount of work. But this might be impossible, as some chores are easier and some are harder; in every single day, some housemates must do the hard chores. 
However, if the same housemates do the same unwanted chores day after day, they will rightfully complain that this \emph{repeated assignment} is unfair. 

A similar problem may arise in different contexts. 
For example, suppose grandpa wants to divide sweets to his grandchildren each weekend. He has to decide who gets the sweets first, who is second, and so on. Each week, the children who are last to get their sweets might feel somewhat unhappy; but if the same children are last in the line week after week, they will naturally feel that the allocation is unfair.
Similar considerations are applicable when assigning seats in classroom to pupils --- a task which may be repeated in different terms.

In all these settings, the same assignment problem repeats many times. Whereas some inequality can be tolerated in every single assignment, we expect it to eventually ``cancel out'' as the same process repeats long enough. We call this problem \emph{Fair Repeated Assignment}. 

In \Cref{sec:setting} below we define the problem formally. 
In the following sections \ref{sec:latin-square}--\ref{sec:ordinal-propc} we present our main results on existence and non-existence of fair repeated assignments.
In \Cref{sec:related} we discuss various other fair allocation settings studied in the literature and compare them to our setting.

\section{Setting}
\label{sec:setting}
There is a set $N$ of $n$ \emph{players}, $N := \{A_1,\ldots,A_n\}$.
There is a set $M$ of $n$ indivisible \emph{items}, $M := [n] = \{1,\ldots,n\}$. We assume that $1$ is the best item (e.g. the easiest chore), $2$ is the second-best item, etc.
This does not lose much generality, as instead of dividing the actual items, we can divide positions in a sequence by which the players pick items, so that the player who receives ``$1$'' is the first to pick an item, the player who receives ``$2$'' picks second, etc. Clearly, an earlier position in the picking sequence is better for each player.%
\footnote{
	This is essentially the ``Identical Order Transformation'' introduced by \citet{bouveret2016characterizing}.
	\citet{cookson2025temporal} prove that this transformation preserves the approximate proportionality condition PROP1.
}

An \emph{assignment} is a bijection between $M$ and $N$; equivalently, a permutation of $M$. 
A \emph{repeated assignment} is a finite sequence of assignments. 
We say that each assignment is implemented in a different \emph{day}.
The multiset of items that a particular player $i$ has after a repeated assignment is called a \emph{bundle}. The bundle of $i$ after day $t$ is denoted $Z_i^t$. 
The multiset of all items allocated up to day $t$ is denoted $Z_N^t := \uplus_{i\in N}Z_i^t = $ the items $1,\ldots,n$, each with multiplicity $t$.

We often present a repeated assignment in a table in which each column represents a day and each row represents an item. The number in column $t$ and row $j$ represents the player who receives item $j$ in day $t$; see \Cref{tab:repeated-assignment-6-6}.
\begin{table}
\begin{center}
\begin{tabular}{cc}
    \begin{minipage}
    {.43\linewidth}
	\begin{tabular}{|c|c|c|c|c|c|}
	\hline
\ifRowIsPlayer
	1&  2&  3&  4&  5&  6\\
	2&  3&  4&  5&  6&  1\\
	3&  4&  5&  6&  1&  2\\
	4&  5&  6&  1&  2&  3\\
	5&  6&  1&  2&  3&  4\\
	6&  1&  2&  3&  4&  5\\
\else 
	1&  6&  5&  4&  3&  2\\
	2&  1&  6&  5&  4&  3\\
	3&  2&  1&  6&  5&  4\\
	4&  3&  2&  1&  6&  5\\
	5&  4&  3&  2&  1&  6\\
	6&  5&  4&  3&  2&  1\\
\fi
	\hline
	\end{tabular}
	\end{minipage}	
		&  
	\begin{minipage}{.57\linewidth}
	\begin{tabular}{|c|c|c|c|c|c|c|c|}
	\hline
	1&  8&  5&  4&  3&  6& 7& 2\\
	2&  7&  6&  3&  4&  5& 8& 1 \\
	3&  6&  7&  2&  1&  8& 5& 4\\
	4&  5&  8&  1&  2&  7& 6& 3\\
	5&  4&  1&  8&  7&  2& 3& 6\\
	6&  3&  2&  7&  8&  1& 4& 5\\
	7&  2&  3&  6&  5&  4& 1& 8\\
	8&  1&  4&  5&  6&  3& 2& 7\\
	\hline
	\end{tabular}
	\end{minipage}	
\end{tabular}
\end{center}
\caption{
Examples of permutation sequences (repeated assignments) of $6$ items (left) and $8$ items (right).
Each column represents an assignment in a single day; each row represents the players who get a single item throughout all days.
\label{tab:repeated-assignment-6-6}
}
\end{table}

For any bundle $Z$ and any $k\leq |Z|$, we denote by $Z[k]$ the $k$-th highest valued item in $Z$, so that $Z[1]$ is the best item in $Z$, $Z[2]$ the second-best, etc.  
The matching between the players and their bundles is called an \emph{allocation}.

\section{Fairness After $n$ Days: latin square}
\label{sec:latin-square}
We are looking for a formal condition that captures the intuitive notion of ``fairness''. 
First, suppose that we are now at the end of day $n$, and we consider the aggregate bundle of $n$ items that is held by each player. We define the allocation at that point as ``fair'' if and only if each player has exactly one copy of each item. It is easy to agree that this condition is sufficient for fairness, as it implies that all players have exactly the same bundle. 
We claim that it is also necessary. This is because 
we make no assumptions on how much exactly item $1$ is better than item $2$, $2$ is better than $3$, etc. Hence, we cannot assume e.g. that a player who received no copy of item $1$, can be compensated by receiving two copies of item $2$. Only allocations that give each item to each player exactly once are guaranteed to be fair.

This condition is satisfied in both examples in \Cref{tab:repeated-assignment-6-6}, where each row is a permutation of $[n]$. 
As each column is also a permutation of $[n]$, the entire repeated assignment table is a \emph{latin square}. There are many ways to construct latin squares, so this requirement is easy to satisfy.

Clearly, in fewer than $n$ days, it is impossible to give all players exactly the same bundle. Suppose item 1 is very valuable (a diamond) and the other items are much less valuable (stones). Then at any day before the $n$-th day, some players (those who did not get the diamond so far) will naturally feel that they did not get their fair share.
~
We could tell these players to wait patiently until day $n$; but  $n$ days may be too much to wait. We would like the allocation to be ``as fair as possible'' after each day --- not only after day $n$.

Below, we explore different ways to define what this fairness means.

\section{Fairness After Every Day: Top-Balance}
\label{sec:top-balance}
A natural fairness requirement is to require that the top items are allocated among the players in a ``balanced'' way: in the first two days, each player should get at least one item from the top half (the highest-valued $\ceil{n/2}$ items); 
in the first three days, each player should get at least one item from the top third (the highest-valued $\ceil{n/3}$ items); and so on.%
\footnote{
	The top-balance condition is asymmetric with respect to time, as it assesses fairness based only on the first $t$ days, and not on the last $t$ days or any other $t$ consecutive days. 
	This asymmetry makes sense when the allocated items are physical items that the players keep, so after $t$ days, each player owns a bundle of $t$ items, which should be sufficiently valuable. 
	A stronger variant of top-balance, that considers simultaneously every sequence of $t$ consecutive days, is briefly discussed at the end of \Cref{sec:ordinal-propc}.
}
 We now state this succinctly.
\begin{definition}
A permutation sequence is called \emph{top-balanced} if after every day $t\in[n]$, every player $i$ has at least one item from the top $\ceil{n/t}$.
Formally,
\begin{align*}
    Z_i^t[1] \leq \ceil{n / t}.
\end{align*}
\end{definition}
In terms of the latin square, top-balance means that, for all $t\in[n]$, the top-left rectangle with $t$ columns and $\ceil{n / t}$ rows must contain all integers $1,\ldots,n$.
The condition is violated by the sequences of \Cref{tab:repeated-assignment-6-6}:
for $n=6$ and $t=2$, the top-left $3\times 2$ rectangle does not contain $5$; for $n=8$ and $t=3$, the top-left $3\times 3$ rectangle does not contain $4$.
Are there other sequences that satisfy it? The answer is yes, and moreover, the condition can be satisfied in combination with the latin square requirement.

\begin{theorem}
\label[theorem]{lem:top-slots}
For every positive integer $n$, there exists an $n\times n$ latin square that is also top-balanced.
\end{theorem}
The proof uses two known results about completing incomplete 
latin squares.  
A \emph{latin rectangle} is a rectangle with $n$ rows and $t<n$ columns, such that each integer in $\{1,\ldots,n\}$ appears exactly once in each column, and at most once in each row.
Marshall Hall \cite{hall1945existence}, using Philip Hall's marriage theorem, proved that any latin rectangle can be extended, one column at a time, to produce a latin square. In our setting, columns correspond to days; we will construct a top-balanced latin-square sequence one day at a time.
~
We will also use Theorem 4 from \cite{EULER2010535} 
to iteratively extend columns from a partial latin rectangle 
to a full latin rectangle.  For convenience, we restate it here 
as \Cref{lem:row-complete}. 

\begin{lemma}[Column completion lemma, Theorem 4 in \cite{EULER2010535}]
\label{lem:row-complete}
Let $L$ be an $n\times t$ latin rectangle where $t<n$.  

Suppose that some $\ell$ distinct labels are entered 
into some $\ell$ distinct rows in column $t+1$.
If we can insert the $n-\ell$ \emph{remaining labels} into the $n-\ell$ \emph{remaining rows} such that the outcome is an $n\times (t+1)$ latin rectangle, then we say that column $t+1$ is \emph{completable}.

The column is completable 
if and only if any subset $I$ of the remaining labels is contained in at most $(n-\ell -|I|)$ of the remaining rows.
    
Equivalently: For any subset $I$ of remaining labels, there are at least $|I|$ remaining rows that do not contain $I$. 
\end{lemma}

\begin{remark}
\label{rem:row-complete}
Since $L$ is a latin rectangle with $t$ columns, any subset $I$ of labels can be contained in at most $t$ rows of $L$ (by the pigeonhole principle).
Hence, a sufficient condition for completability is that $n-\ell-|I|\geq t$.
Moreover, the condition holds trivially when $|I|>t$, so it is sufficient to check the condition for $|I|\leq t$. Hence, it is sufficient to prove that $n-\ell-t\geq t$.
\end{remark}

\begin{proof}[of Theorem \ref{lem:top-slots}]
First, we show how to define the assignment for the first half of the days, namely for columns $t\in \{1,\ldots \ceil{\frac{n}{2}}\}$,
and for the top portion of each of these columns (representing the assignment of the top items).
After the top-balance condition is satisfied for all these columns, 
we show how to append the remaining columns such that the 
entire sequence is a latin square and top-balanced. 

\paragraph{\textbf{Step 1: Constructing a top-balanced partial latin rectangle.}}
We build the top-balanced sequence and partial latin square inductively in the columns $t$.
We want to establish the following two claims inductively in $t$. 
\begin{enumerate} 
\item The set of labels in the top-$\ceil{\frac{n}{t}}$ rows in columns $1,\ldots,t$ contains $[n]$, and 
\item The labels in the top-$(\ceil{\frac{n}{t}} - 1
)$ rows in columns $1,\ldots,t$ are all unique. 
\end{enumerate}

For the induction base, column $t=1$ is filled by $1,\ldots,n$. Column $t=2$ is filled by $\lceil \frac{n}{2}\rceil +i$ for $1\leq i \leq \lfloor \frac{n}{2}\rfloor$,
and then by $i$ for 
$1\leq i \leq \lceil \frac{n}{2}\rceil$.
This construction clearly satisfies both claims.

Suppose we have filled columns $1,\ldots,t$ for $2\leq t<n$ such that claims (1) and (2) hold.
We show how to progress to the next column $t+1$.  There are two cases to consider.

\underline{Case \#1:} 
$\ceil{\frac{n}{t+1}} \leq \ceil{\frac{n}{t}}-1$. Then by the induction hypothesis the labels in the top-$\ceil {\frac{n}{t+1}}$ rows and first $t$ columns are unique.
In column $t+1$, fill the top $\ceil {\frac{n}{t+1}}-1$ cells by labels not in this collection;
This is possible since
$(t+1)(\ceil{\frac{n}{t+1}} -1)< n$. 
The labels in the top-
$(\lceil \frac{n}{t+1}\rceil -1)$ rows
in the first $t+1$ columns are still unique, so claim (2) is satisfied.

Since $(t+1)\lceil \frac{n}{t+1} \rceil \geq n$, 
at most one label does not appear in a top-$\ceil {\frac{n}{t+1}}$ row yet; put that label in cell $\ceil {\frac{n}{t+1}}$. Now,
all $n$ labels appear in a top-$\ceil{\frac{n}{t+1} }$ row in the first $t+1$ columns, so claim (1) is satisfied too.

\underline{Case \#2:} $\lceil \frac{n}{t+1} \rceil = \lceil \frac{n}{t} \rceil$.
Then by the induction assumption all labels already appear in a top-$\lceil \frac{n}{t+1} \rceil$ row in columns $1,\ldots,t$, so claim (1) is already satisfied.
We only have to ensure uniqueness.
Take any labels that did not appear in a top $(\lceil \frac{n}{t+1} \rceil - 1)$ row yet, and put them in the top cells of column $t+1$.  
Since $(t+1)(\lceil \frac{n}{t+1} \rceil -1)<n$, 
all labels in one of the top 
$\lceil \frac{n}{t+1} \rceil -1$ rows in the first $t+1$ 
columns are unique, and claim (2) is satisfied too.

\paragraph{\textbf{Step 2: Completing the columns to a full latin rectangle.}}
For each $t \in \{1,\ldots, \ceil{\frac{n}{2}}-1 \}$, we will show how to complete column $t+1$, such that the rectangle consisting of days $1,\ldots,t+1$ is a latin rectangle. 
Step 1 already handles $t=1$, so we can assume $t\geq 2$.
~
In \Cref{sec:full-balance}, we will present a fairness notion stronger than both latin square and top-balance; in \Cref{thm:fully-balanced}, we will show that it can be satisfied for all $n\leq 11$.
Hence, we can assume here that $n\geq 12$.

To complete column $t+1$, we apply  \Cref{lem:row-complete} with $\ell$ being the number of cells filled in Step $1$.  Note, we have that $\ceil{\frac{n}{t+1}}-1 \leq \ell \leq \ceil{\frac{n}{t+1}}$.
~
If $t= 2$, 
$$
n-\ell-t \geq 
n-\ceil{\frac{n}{t+1}} -t > 
n-\Big(\frac{n}{3}+1\Big)-2=\frac{2n}{3}-3\geq 5>t,$$
which is sufficient by \Cref{rem:row-complete}.
Similarly, if $t= 3$, then 
$$
n-\ell-t \geq 
n-\ceil{\frac{n}{t+1}} -t > 
n-\Big(\frac{n}{4}+1\Big)-3=\frac{3n}{4}-4\geq 5>t.$$
Now, for $n\geq 12$, $t\geq 4$ and $t\leq \big\lceil \frac{n}{t+1}\big\rceil$, 
$$
n-\ell-t \geq 
n-2\ceil {\frac{n}{t+1}} > \frac{3n}{5} - 2 
> \frac{n}{5}+1 \geq \ceil{\frac{n}{t+1}} \geq t.$$
Again, by \Cref{rem:row-complete}, \Cref{lem:row-complete} is satisfied and we are able to complete column $t+1$.  
~
Finally, we handle  $t>\ceil{\frac{n}{t+1}}$. We consider two cases.

\underline{Case \#1:} $n$ is even (and $n\geq 12$).
Recall $\ceil{\frac{n}{t+1}}-1\leq \ell \leq \ceil{\frac{n}{t+1}}$.
The construction of Step 1 ensures that all $n$ labels appear at least once in the top $\ceil{\frac{n}{t+1}}$ rows of the first $t+1$ columns.
Hence, any set $I$ of labels that do not already appear in the top $\ell$ rows of column $t+1$, must appear in the top $\ceil{\frac{n}{t+1}}$ rows of the first $t$ columns at least once.  If $|I|=t$, since $t>\ceil{\frac{n}{t+1}}$, there exists an element $a\in I$ such that $a$ appears in the top $\ell$ rows and the first $t$ columns.  Thus, the element $a$ can only appear at most $t-1$ times in the bottom $n-\ell$ rows.  Hence, to apply \Cref{lem:row-complete}, it is sufficient to show a slightly weaker condition than that of \Cref{rem:row-complete}: $n-\ell-|I| \geq t-1.$  
As $|I|\leq t$, it is sufficient to show that 
$n-\ell-t \geq t-1.$ 

On the other hand, if $|I|\leq t-1$, then it is sufficient to show $n-\ell -(t-1) \geq t$ which also implies it is sufficient to show $n-\ell-t \geq t-1.$  

Substituting an upper bound for $\ell$ and rearranging terms leads to the following sufficient condition:
$$n-\Big( \frac{n}{t+1}+1\Big)-t - (t-1) = n-\frac{n}{t+1}-2 t \geq 0.$$
We know for $n\geq 12$ and $t=1$, $n-\frac{n}{t+1}-2 t>0$. 
Solving $n-\frac{n}{t+1}-2t = 0$ gives
$t=\frac{n}{2}-1$ or $t=0$ (recall that $n$ is even). 
By Bolzano's Theorem \cite{bolzano1817rein}, this implies the \Cref{lem:row-complete} condition holds for $0\leq t\leq \frac{n}{2}-1$.  
Hence, we can complete all columns up to column $\frac{n}{2}$.

\underline{Case \#2:} $n$ is odd (and $n\geq 13$).
By the argument of Case \#1, we can complete 
column $t+1$ for all $t\leq \frac{n}{2}-1$,
which for odd $n$ means $t\leq \frac{n-3}{2}$.
Thus, the first $\frac{n-1}{2}$ columns can be completed.

It remains to complete column $\frac{n+1}{2} = \ceil{\frac{n}{2}}$.
As $n-1$ labels have already appeared in the top two rows and the first $\frac{n-1}{2}$ columns, 
the construction of Step 1 has to put only the single remaining label at the top row of column $\frac{n+1}{2}$, so we can take $\ell=1$. As $t=\frac{n-1}{2}$, we get
$$
n - \ell - t = n - 1 - \frac{n-1}{2} = \frac{n-1}{2} = t.
$$
The condition in \Cref{rem:row-complete} is satisfied, so column $\frac{n+1}{2}$ can be completed. 

\paragraph{\textbf{Step 3: Completing the rectangle to a full latin square.}}
At this point, we have a top-balanced latin rectangle with $\ceil{\frac{n}{2}}$ columns.
Apply Hall's theorem \citep{hall1945existence} to complete 
the final $\floor{\frac{n}{2}}$ columns and obtain a latin square. 

Since $\floor{\frac{n}{2}}$ labels have not appeared at row 1 in the first $\ceil{\frac{n}{2}}$ columns, 
these labels must appear in the first row 
in the last $\floor{\frac{n}{2}}$ columns to obtain a latin square. This implies that the resulting latin square remains top-balanced. 
\qed
\end{proof}

\section{Stronger Fairness After Every Day: Full-Balance}
\label{sec:full-balance}
The top-balance condition considers only the top items for each player. The allocations of other items can be arbitrary, and thus unfair.
The condition can be strengthened as follows:
\begin{itemize}
	\item In the first two days, each player should receive at least one item from the top $\ceil{n/2}$ items --- this condition is the same as in top-balance.
	\item In the first three days, each player should receive at least one item from the top $\ceil{n/3}$ items, \emph{and at least two items from the top $\ceil{2 n/3}$ items} --- the latter condition is new.
	\item In general, in the first $t$ days, for every $j\in\{1,\ldots,t\}$, each player should receive at least $j$ items from the highest-valued $\ceil{j n/t}$ items.
\end{itemize}
All these conditions are formalized succinctly below.
\begin{definition}
A permutation sequence is called \emph{fully-balanced} if after every day $t\in[n]$, the bundle $Z_i^t$ of every player $i$ satisfies the following  for every $j\in[t]$:
\begin{align}
	\label{eq:bal}
	Z_i^t[j] \leq \ceil{j n / t}.
\end{align}
\end{definition}

We denote by $\propcond(t,j)$ the condition \eqref{eq:bal} for some particular $t\in[n]$ and $j\in[t]$. 
So a repeated assignment is fully-balanced if it satisfies $\propcond(t,j)$ for every day $t\in[n]$ and item $j\in[t]$.

\begin{remark}
    (a) $\propcond(t=n,j)$ implies that, after $n$ days, the $j$-best highest item of each player is item $j$; this implies that each player has exactly one copy of each item, which is the Latin-square condition of \Cref{sec:latin-square}.
    
    (b) $\propcond(t,j=1)$ is exactly the top-balance condition of \Cref{sec:top-balance}.

    (c) $\propcond(t=1,\cdot)$ and $\propcond(t,j=t)$ hold trivially, so in fact we only have to check $t\in\{2,\ldots,n\}$ and $j\in\{1,\ldots,t-1\}$.
\end{remark}

Do fully-balanced sequences exist for all $n$?

It is easy to find fully-balanced sequences for $n \in \{3,4,5,6\}$; the readers may want to try it themselves before looking at \Cref{tab:n=3456}. To verify that a permutation sequence is fully-balanced, one should check, for each $t\in[n]$ and $j\in [t]$, that the top-right rectangle with $t$ columns and $\ceil{jn/t}$ rows contains each integer in $1,\ldots,n$ at least $j$ times.
For example, for $n=6$ and $t=4$, we check that the $2\times 4$ top-left rectangle contains $1,\ldots,6$ at least once; 
 the $3\times 4$ top-left rectangle contains $1,\ldots,6$ at least twice; and the $5\times 4$ top-left rectangle contains $1,\ldots,6$ at least three times.
\begin{table}
\begin{center}
	\begin{tabular}{cccc}
		\begin{minipage}{.17\linewidth}
\begin{tabular}{|c|c|c|}
	\hline
	1 & 2 & 3 \\
	2 & 3 & 1 \\
	3 & 1 & 2 \\
	\hline
\end{tabular}
		\end{minipage}	
		&  
		\begin{minipage}{.22\linewidth}
\begin{tabular}{|c|c|c|c|}
	\hline
	1 & 4 & 3 & 2 \\
	2 & 3 & 4 & 1 \\
	3 & 2 & 1 & 4 \\
	4 & 1 & 2 & 3 \\
	\hline
\end{tabular}
		\end{minipage}	
		&  
\begin{minipage}{.27\linewidth}
	\begin{tabular}{|c|c|c|c|c|}
		\hline
\ifRowIsPlayer
		1 & 3 & 5 & 4 & 2\\
		2 & 4 & 1 & 5 & 3\\
		3 & 5 & 2 & 1 & 4\\
		4 & 1 & 3 & 2 & 5\\
		5 & 2 & 4 & 3 & 1\\
\else 
		1 & 4 & 2 & 3 & 5\\
		2 & 5 & 3 & 4 & 1\\
		3 & 1 & 4 & 5 & 2\\
		4 & 2 & 5 & 1 & 3\\
		5 & 3 & 1 & 2 & 4\\
\fi
		\hline
	\end{tabular}
\end{minipage}	
		&  
\begin{minipage}{.33\linewidth}
	\begin{tabular}{|c|c|c|c|c|c|}
		\hline
\ifRowIsPlayer
		1 & 4 & 5 & 3 & 2 & 6\\
		2 & 5 & 3 & 6 & 1 & 4\\
		3 & 6 & 1 & 5 & 4 & 2\\
		4 & 1 & 6 & 2 & 3 & 5\\
		5 & 2 & 4 & 1 & 6 & 3\\
		6 & 3 & 2 & 4 & 5 & 1\\
\else 
		1 & 4 & 3 & 5 & 2 & 6\\
		2 & 5 & 6 & 4 & 1 & 3\\
		3 & 6 & 2 & 1 & 4 & 5\\
		4 & 1 & 5 & 6 & 3 & 2\\
		5 & 2 & 1 & 3 & 6 & 4\\
		6 & 3 & 4 & 2 & 5 & 1\\
\fi
		\hline
	\end{tabular}
\end{minipage}	
	\end{tabular}
\end{center}
\caption{
	Fully-balanced permutation sequences for $n\in\{3,4,5,6\}$.
\label{tab:n=3456}
}
\end{table}
We were able to handcraft sequences for $n\leq 8$.  For larger values of $n$, we implemented a backtracking algorithm, and found such sequences for all $n\leq 11$.
\Cref{app:fully-balanced-sequences} shows the fully-balanced sequences that we found for $7\leq n \leq 11$. 
\begin{theorem}
	\label{thm:fully-balanced}
    For all $n\leq 11$, a fully-balanced permutation sequence exists.
\end{theorem}

Twelve (12) is the first natural number that does not admit 
a fully-balanced sequence.  Below we show there 
are infinitely many natural numbers with no fully-balanced sequence, including $n=12$. 
\begin{proposition}
	\label[proposition]{lem:n=6k}
There is no fully-balanced sequence of length $n=6k$ for $k\geq 2$. 
\end{proposition}
\begin{proof}
Let $A_i$ represent the player who goes $i^{th}$ in the first day. Let $B_i$ represent the player who goes $i^{th}$ in the second day; $C_i$ the player who goes $i^{th}$ in the third day and $D_i$ the player who goes $i^{th}$ in the fourth day. We will show that balance must break down within the fourth day.  Notice that for $k\geq 2$, $\ceil{\frac{6k}{4}} < 2k$.  
Here is the list of implications that prevent balance:
\begin{enumerate}
\item \label{ps-1}
$\rangeA{3k+1}{6k} = \rangeB{1}{3k}$
by $\propcond(2,1)$. Each player that appears in the second half in the first day must appear in the first half in the second day, since $\ceil{\frac{6k}{2}}= 3k$. 

\item \label{ps-2}
$\rangeA{1}{3k}=\rangeB{3k+1}{6k}$.
This follows by \eqref{ps-1}.

\item \label{ps-3}
$\rangeC{1}{2k}=\rangeA{2k+1}{3k} \cup \rangeB{2k+1}{3k}$ by $\propcond(3,1)$. Since $\ceil{\frac{6k}{3}}=2k$, each player must go at least once in the top $2k$ positions after the first $3$ days.  But, $\rangeA{2k+1}{3k} \cup \rangeB{2k+1}{3k}$ are $2k$ distinct players that have not gone in the top $2k$ positions after the first $2$ days. 

\item \label{ps-4}
$\{ C_{2k}\} \subset \rangeD{1}{3k}$ by $\propcond(4,1)$.
Let $z :=\ceil{\frac{6k}{4}}$. Since $z < 2k$, then player $C_{2k}$ did not get any of the top $z$ items in any of the first three days.  Thus, $C_{2k}\in \rangeD{1}{z}$ 
(Moreover, $\rangeC{z+1}{2k} \subseteq \rangeD{1}{z}$).

\item \label{ps-5} 
$\rangeC{3k+1}{6k} \subseteq \rangeD{1}{3k}$ by $\propcond(4,2)$. 
Each player is in the first $3k$ positions in the first or second day, but not both days.  This follows from (\ref{ps-1}) and (\ref{ps-2}).  Since $\frac{2(6k)}{4}=3k$, then each player goes in the top $3k$ positions at least twice in the first $4$ days.  But, each of $C_{3k+1},\ldots ,C_{6k}$ have gone once in the first $3$ days.  
\end{enumerate}
But there is no room for all of $C_{2k},C_{3k+1},\ldots ,C_{6k}$ to go in the top $3k$ positions in the fourth day. That is, 
$$|\{C_{2k},C_{3k+1},C_{3k+2},\ldots ,C_{6k}\}|=|\{D_1,D_2,\ldots ,D_{3k}\}|+1$$
which leads to a contradiction. 
 \end{proof}

\Cref{lem:n=6k} can be extended to all $n$. The proof varies with the remainder of $n$ divided by $6$, so there are six different proofs (including \Cref{lem:n=6k}).
The table below shows the lower bounds for each remainder $j$ in $\{0,1,2,3,4,5\}$.  The proofs are given in \Cref{app:no-balanced-sequence}.
\begin{table}[!ht]
    \centering 
\begin{tabular}{|c||c|c|c|c|c|c|}
\hline 
$n$ & $6k$ & $6k+1$ & $6k+2$ & $6k+3$ & $6k+4$ & $6k+5$ \\ 
\hline
\mbox{bounds on $k$} & $k\geq 2$ & $k\geq 11$ & $k\geq 3$ & $k\geq 8$ & $k\geq 4$ & $k\geq 9$ \\ 
\hline
\mbox{bounds on $n$} & $n\geq 12$ & $n\geq 67$ & $n\geq 20$ & $n\geq 51$ & $n\geq 28$ & $n\geq 59$ \\ 
\hline 
\end{tabular}
~\\
~\\
\caption{Values of $n=6k+j$ with no fully-balanced sequence}\label{tab:no-full-balance}
\end{table}
\Cref{tab:no-full-balance} shows there are at most finitely many values $n$ 
which admit a fully-balanced sequence.  In particular, the following 
theorem follows from \Cref{tab:no-full-balance}. 
\begin{theorem}
\label[theorem]{thm:finitely-many-balanced}
There exist finitely many $n$ such that there is 
a fully-balanced sequence of length $n$.  Moreover, 
there is no fully-balanced sequence for $n> 61$. 
\end{theorem}
\begin{proof}
Since every natural number can be expressed uniquely 
in the form $n=6k+j$ for $j=0,1,2,3,4,5$, then we 
can list the natural numbers that fall under the bounds 
in \Cref{tab:no-full-balance}.  It is straightforward to check that 
$n=61$ is the largest natural number which possibly 
has a fully-balanced sequence. 
See the \Cref{app:no-balanced-sequence} for a proof of \Cref{tab:no-full-balance} and this theorem. 
\end{proof}

\section{Valuations and Proportionality: Weak Full-Balance}
As the concept of a fully-balanced sequence is too strong, we would like to relax it. 
To do so (and also to link our balance concepts to the literature on fair allocation), we assume that each item $j\in[n]$ has a \emph{value} $v_j\in\mathbb{R}$, such that $v_1 > v_2 > \cdots > v_n$.%
\footnote{
\label{ftn:different-valuations}
The literature on fair division studies a more general setting, in which different players may have different rankings and different valuations over the items. In \Cref{app:het-valuations} we prove that our main theorem holds in this more general setting too.
}
The total value of any bundle (multiset of items) $Z$ is denoted $v(Z) := \sum_{k\in Z} v_k$. 

A natural fairness concept is that each player should receive a bundle of items that is worth at least $1/n$ of the total value, that is, $v(Z_i^t)\geq v(Z_N^t)/n$.%
\footnote{
The phrase ``at least'' comes from the more general setting (\cref{ftn:different-valuations}). In the case of a common valuation, proportionality implies that each player receives exactly $1/n$ of the total value.
}
This concept is prominent both in law and in the literature on fair division ever since its inception \citep{steinhaus1948problem}. Nowadays, it is called  \emph{proportionality} \citep{brams1996fair,robertson1998cake}. 
We require a strong variant of proportionality, which holds not only for specific valuations but for \emph{any} valuations that satisfy the condition $v_1 > \cdots > v_n$. We call this variant \emph{ordinal proportionality}, as it is based only on the order among items rather than on their exact values.%
\footnote{Ordinal fairness 
	is also known as \emph{necessary fairness} or \emph{stochastic-dominance fairness (sd-fairness)}.
	Ordinal fairness in item allocation has been studied e.g. by  \citet{bouveret2010fair,pruhs2012divorcing,aziz2015fair,segal2020fair,brams2022two,brams2026fair}.
}

Given the latin square constraint (\Cref{sec:latin-square}), the allocation after day $n$ is ordinally proportional. 
This is because the bundle of every player $i$ satisfies $v(Z_i^n) = \sum_{j\in M} v_j$ and the total value of all items is $v(Z_N^n) = n\cdot \sum_{j\in M} v_j$.
However, before day $n$ it is impossible to guarantee ordinal proportionality. For example, if $v_1$ is very high and the other item values are very low, then at any day before the $n$-th day, the value of some players (those who did not get item 1 so far) will be below the proportionality threshold.

A natural relaxation of proportionality for indivisible items is \emph{proportionality up to $c$ items}, abbreviated PROP$c$ \citep{conitzer2017fair,chakraborty2021weighted,gourves2021fairness,bu2023fair,kahana2026perpetual}. It means that, for each player $i$, if we add to $i$'s bundle some $c$ items not allocated to $i$ yet, then $i$'s value will be at least $1/n$ of the total value. This definition implicitly assumes that all items are goods (have positive values); we use a more symmetric definition, that applies equally to items with positive or negative values: instead of just adding $c$ items, we both add and remove the same number of items. 
\begin{definition}
	\label{def:propc}
Let $c\geq 1$ be an integer.

(a) A repeated assignment is called \emph{perpetually PROP$c$ for valuation $v$} if for every time $t\in[n]$ and player $i$, there is a set $C'\subseteq N\setminus Z^t_i$ and a set $C''\subseteq Z^t_i$, with $|C''|=|C'|\leq c$, such that $v(Z^t_i\setminus C'' \cup C')\geq v(Z^t_N)/n$.

(b) A repeated assignment is called \emph{perpetually ordinally PROP$c$} if it is perpetually PROP$c$ for every valuation $v$ satisfying $v_1 > \cdots > v_n$.
\end{definition}
Note that, when all values are positive, \Cref{def:propc}(a) implies the standard definition of PROP$c$, as removing items only makes the value smaller.

The sequence for $n=6$ at \Cref{tab:repeated-assignment-6-6} fails ordinal PROP1 already after the second day. This is because the total value allocated up to this time is $2\cdot (v_1+v_2+v_3+v_4+v_5+v_6)$, whereas the bundle of player 5 is $\{5,6\}$. Even if we add the most valuable item $1$, the resulting bundle value is $v_1+v_5+v_6$, which might be smaller than $1/6$ of the total value (e.g. if $v_5=v_6=0$ and $v_1=\cdots=v_4=1$).
Hence, the repeated assignment at \Cref{tab:repeated-assignment-6-6} is not perpetually ordinally PROP1.
~
Below we relate full balance to perpetual proportionality.
\begin{theorem}
\label[theorem]{thm:full-balance-implies-ordinal-prop1}
A fully-balanced permutation sequence generates a perpetually ordinally PROP1 repeated assignment.
\end{theorem}
\Cref{thm:full-balance-implies-ordinal-prop1} follows directly from \Cref{thm:weak-full-balance-implies-ordinal-prop1} below, which shows that ordinal PROP1 follows from an a slightly weaker balance condition.
\begin{definition}
A permutation sequence is called \emph{weakly fully-balanced} if for all $t\in[n]$, the bundle $Z_i^t$ of every player $i$ satisfies the following  for all $j\in[t]$:
\begin{align}
	\label{eq:weakbal}
	Z_i^t[j] \leq \floor{j n / t+1}.
\end{align}
\end{definition}
We call this condition $\weakpropcond(t,j)$. 

Note that $\weakpropcond(t,j)$ is weaker than $\propcond(t,j)$ only when $jn/t$ is an integer; otherwise they are equivalent.

\begin{theorem}
\label[theorem]{thm:weak-full-balance-implies-ordinal-prop1}
A weakly fully-balanced permutation sequence generates a perpetually ordinally PROP1 repeated assignment.
\end{theorem}

\begin{proof}
We fix a day $t$, and prove that the cumulative allocation up to day $t$ satisfies ordinal PROP1.

So far, $n t$ items have been allocated, so $|Z_N^t| = n t$. Note that the multiset $Z_N^t$ contains exactly $t$ clones of every item in $1,\ldots,n$.

Fix a player $i$. Let us add to $i$'s bundle $Z_i^t$ a copy of the highest-valued item (1), and remove the least-valuable item from $i$'s bundle, so that $i$'s new bundle still has $t$ items overall. Denote this new bundle $C_i$.
To show PROP1, it is sufficient to show that the $v(C_i)\geq v(Z_N^t)/n$.
~
Construct a new multiset $D_i$, containing $n$ clones of each item in $C_i$; now, it is sufficient to show that $v(D_i)\geq v(Z_N^t)$. We do this by showing a bijection $f: D_i \to Z_N^t$, such that $x \leq f(x)$ for all $x\in D_i$; hence $v(x)\geq v(f(x))$ for all $x\in D_i$.

We first show the bijection for the special case that $n/t$ is an integer.
\begin{itemize}
\item We first map the $n$ clones of item 1 in $D_i$ (the added item).  We map them to the $t$ clones of items $1,2,\ldots,n/t$ in $Z_N^t$.

\item Next, we map the $n$ clones of the best item held by $i$, denoted $Z_i^t[1]$. By $\weakpropcond(t,j=1)$, $Z_i^t[1]\leq \floor{n/t+1} = n/t + 1$. Hence, we can map it to the $t$ clones of items $n/t+1,2,\ldots,2 n/t$ in $Z_N^t$.

\item Similarly, by $\weakpropcond(t,j)$ for all $j\in[t-1]$, 
$Z_i^t[j]\leq j n /t + 1$; hence, we can map the $n$ clones of $Z_i^t[j]$ to the $t$ clones of items $j n/t+1,2,\ldots,(j+1) n/t$ in $Z_N^t$.
\end{itemize}
When $n/t$ is not an integer, the mapping works as follows:
\begin{itemize}
\item The $n$ clones of item 1 in $D_i$ are mapped to the $t$ clones of items $1,2,\ldots,\floor{n/t}$ in $Z_N^t$, as well as to some $n \% t$ clones of item $\floor{n/t+1}$ (here $\%$ denotes the integer remainder operation).
	
\item The $n$ clones of $Z_i^t[1]$ in $D_i$, which is at most $\floor{n/t+1}$ by  $\weakpropcond(t,j=1)$, are mapped to the remaining $t-n \% t$ clones of item $\floor{n/t+1}$ in $Z_N^t$, as well as to the $t$ clones of $\floor{n/t+2}$ and some lower-valued items. 
Overall, the $2n$ clones of $1$ and $Z_i^t[1]$ in $D_i$ are mapped to the $t$ clones of items $1,\ldots,\floor{2n/t}$ as well as to some clones of $\floor{2n/t+1}$ in $Z_N^t$.

\item Similarly, for all $j\in[t-1]$, the $(j-1)n$ clones of items $1, Z_i^t[1],\ldots,Z_i^t[j-1]$ in $D_i$ are mapped to the $t$ clones of items $1,\ldots, \floor{j n/t}$, as well as to some clones of $\floor{j n / t+1}$ in $Z_N^t$.
By $\weakpropcond(t,j)$ $Z_i^t[j]\leq \floor{j n /t + 1}$, so the clones of $Z_i^t[j]$ can be mapped to the remaining clones of $\floor{j n /t + 1}$ and to lower-valued items.
\end{itemize}

\begin{table}
	\begin{center}
		\begin{tabular}{cc}
			\begin{minipage}{.5\linewidth}
				\begin{tabular}{|p{5mm}|p{5mm}|p{5mm}|p{5mm}|p{5mm}|p{5mm}|}
					\hline
					1 & 1 & 1 & 1 & 1 & 1\\
					\hline
					2 & 2 & 2 & 2 & 2 & 2 \\
					\hline
					3 & 3 & 3 & 3 & 3 & 3 \\
					\hline
					4 & 4 & 4 & 4 & 4 & 4\\
					\hline
					5 & 5 & 5 & 5 & 5 & 5\\
					\hline
					6 & 6 & 6 & 6 & 6 & 6\\
					\hline
					7 & 7 & 7 & 7 & 7 & 7 \\
					\hline
					8 & 8 & 8 & 8 & 8 & 8\\
					\hline
					9 & 9 & 9 & 9 & 9 & 9\\
					\hline
					10 & 10 & 10 & 10 & 10 & 10\\
					\hline
				\end{tabular}
                \\
                \centering
                $Z_N^t$
			\end{minipage}	
			&  
			\begin{minipage}{.5\linewidth}
				\begin{tabular}{|p{5mm}|p{5mm}|p{5mm}|p{5mm}|p{5mm}|p{5mm}|}
					\hline
					1  & 1 & 1 & 1 & 1 & 1\\
					\hline
					1 &  1 & 1 & 1 & 2 & 2\\
					\hline
					2 &  2 & 2 & 2 & 2 & 2 \\
					\hline
					2 &  2 & 4 & 4 & 4 & 4 \\
					\hline
					4 &  4 & 4 & 4 & 4 & 4 \\
					\hline
					6 &  6 & 6 & 6 & 6 & 6\\
					\hline
					6 & 6  & 6  & 6 & 7 & 7\\
					\hline
					7 & 7  & 7  & 7 & 7 & 7 \\
					\hline
					7 & 7  & 9  & 9 & 9 & 9 \\
					\hline
					9 & 9  & 9  & 9 & 9 & 9 \\
					\hline
				\end{tabular}
                \\
                \centering
                $D_i$
			\end{minipage}	
		\end{tabular}
	\end{center}
	\caption{
		\label{tab:weakprop-implies-prop1}
		Example for the proof of \Cref{thm:weak-full-balance-implies-ordinal-prop1}.
		Here $n=10,t=6$.
	}
\end{table}

For illustration, \Cref{tab:weakprop-implies-prop1} shows the items in $Z_N^t$ (left), and the items $1$ and $\floor{j n / t + 1}$ for $j\in[t-1]$ in $D_i$ (right) for $n=10,t=6$. One can see that each item in the $D_i$ table is at least as good as the item at the same location in the $Z_N^t$ table, as required.
\end{proof}

We proved in \Cref{sec:top-balance} that there is no fully-balanced sequence for $n=12$. In contrast, we could find a weakly-fully-balanced sequence for $n=12$ does exist;
see \Cref{table-wprop-12} in \Cref{app:fully-balanced-sequences}. 

\begin{corollary}
    A perpetually ordinally PROP1 repeated assignment exists for all $n\leq 12$.
\end{corollary}

Unfortunately, the impossibility result for a fully-balanced sequence extends to weakly-fully-balanced sequences, though for larger values of $n$, as shown in \Cref{tab:no-weak-full-balance}. The proof is in \Cref{prop:no-weakprop-all-residues} in \Cref{app:no-weak-balanced-sequence}.

\begin{table}[!ht]
	\centering 
	\begin{tabular}{|c||c|c|c|c|c|c|}
		\hline 
		$n$ & $6k$ & $6k+1$ & $6k+2$ & $6k+3$ & $6k+4$ & $6k+5$ \\ 
		\hline
		\mbox{bounds on $k$} & $k\geq 19$ & $k\geq 11$ & $k\geq 16$ & $k\geq 12$ & $k\geq 17$ & $k\geq 9$ \\ 
		\hline
		\mbox{bounds on $n$} & $n\geq 114$ & $n\geq 67$ & $n\geq 98$ & $n\geq 75$ & $n\geq 106$ & $n\geq 59$ \\ 
		\hline 
	\end{tabular}
	~\\
	~\\
	\caption{Values of $n=6k+j$ with no weakly fully-balanced sequence
	\label{tab:no-weak-full-balance}
	}
\end{table}

\section{Ordinal PROP$c$: Open Question}
\label{sec:ordinal-propc}
Can we satisfy perpetual ordinal PROP$c$ for some $c\geq 2$? For this, it is sufficient to satisfy the following condition for all players $i\in N$, after every day $t \in [n]$, for every $j\in[t]$:
\begin{align}
	\label{eq:weak2propcond}
	Z_i^t[j] \leq \floor{(j+c-1) n/t + 1}.
\end{align}
We call this condition $\weakpropcond(c,t,j)$;
the condition previously called 

\noindent 
$\weakpropcond(t,j)$ is actually $\weakpropcond(1,t,j)$.
For $c=2$, Condition \eqref{eq:weak2propcond} is non-trivial only for $t\geq 3$ and $j\leq n-2$, for example:
\begin{itemize}
\item In the first three days, each player gets at least one item from the top $\floor{2n/3+1}$;
\item In the first four days, each player gets at least one item from the top $\floor{2n/4+1}$ and at least two items from the top $\floor{3n/4+1}$;
\end{itemize}

We call a permutation sequence \emph{$c$-balanced} if it satisfies $\weakpropcond(c,t,j)$ for all $t \in [n]$ and all $j\in[t]$.
A $c$-balanced permutation sequence generates a perpetually ordinally PROP$c$ assignment; the proof is similar to \Cref{thm:weak-full-balance-implies-ordinal-prop1}.


\begin{open}
Is there, for every $n\geq 1$, a $2$-balanced permutation sequence?
A $c$-balanced permutation sequence for some constant $c\geq 2$?
\end{open}
We have preliminary results showing that a $c$-balanced permutation exists for $c\in O(\log n)$, but strengthening these results to $c\in O(1)$ remains challenging.


Another future work direction involves a different possible strengthening of the top-balance condition: every player should get at least one item from the top $\ceil{n/t}$, not only at the first $t$ days, but also on every sequence of $t$ consecutive days. This condition cannot be satisfied already for $n=12$. This is because, in the first three days, there must be three pairwise-disjoint subsets of four players each, who receive the top four items. In particular, the player who gets item 4 on day 3 has not received any top-3 item so far, and thus must receive one of the top 3 items in day 4. This player gets one of the top 6 items both in day 3 and in day 4. Hence, some other player does not get any of the top 6 items in these two consecutive days.
Is there a weaker variant of the top-balance condition, that can be guaranteed to hold on every sequence of $t$ consecutive days?

\section{Related Work}
\label{sec:related}
\subsection{Fair division}
The literature on fair division is very large, with new papers appearing every several weeks. See \citet{amanatidis2022fair,aziz2022algorithmic,nguyen2023complexity,lang2024fair} for some recent surveys.
Most works on fair allocation consider a one-round division, but in recent years there is a growing interest in multiple-round division. Below we present several settings that are related to ours. They are  ordered by increasing difficulty --- from the ``easiest'' ones (where strongest fairness guarantees are possible) to the ``hardest'' ones (where only weak fairness guarantees are possible).

\paragraph{Repeated fair assignment.}
This setting is most closely related to ours, as it involves a single assignment problem (of exactly one item to each player) that is repeated every day.
It was recently studied by \citet{micheel2024fairness}.
They studied the computational complexity of three separate fairness notions:
(a) \emph{Mirrored envy}: if some player envies another player a certain number of times, then the latter envied them the same number of times.
(b) \emph{Equal treatment of equals}: players with the same preferences should get exactly the same final bundle. 
(c) \emph{Minimizing the number of times a player is envious}.

\citet{caragiannis2024repeatedly} study a generalization of repeated fair assignment, in which the value of a player for an item may depend on the number of rounds the item was previously used by the same player. In this setting, even obtaining fairness at the last day becomes challenging, as the round-robin technique might not work. They study the run-time of computing an allocation that maximizes the sum of values (``utilitarian''), and present some special cases in which it is possible to ensure EF1 (envy-free up-to one item) at the last day.

\citet{lim2026fairness} study the computational problem of finding a repeated allocation that maximizes the smallest utility overall, or every individual day. This is NP-hard in general, but can be solved efficiently in some special cases.

\paragraph{Repeated fair allocation.}
This setting is more general than ours, as it allows the number of items allocated each day to be different than the number of players, and do not require that each player receives exactly one item per day \citep{igarashi2024repeated}. Naturally, the attainable fairness guarantees are more restricted. In particular:
(a) When $n\geq 3$ and the number of days is fixed in advance (as in our setting), only  last-day fairness is guaranteed, and only when the number of days is a multiple of $n$; there are no guarantees on fairness after each day.
(b) When $n = 2$ and the number of days is fixed in advance and is even, there are repeated allocations that are envy-free after the last day, and in addition, the allocation in each individual day satisfies a weak variant of EF1.
(c) When the number of rounds can be chosen by the algorithm, 
there are sequences that are simultaneously last-day envy-free and Pareto-efficient, and individual-day PROP1. 

\citet{cookson2025temporal} strengthen their results to \emph{ordinal fairness}, as in our setting. They show polynomial-time algorithms that guarantee the following:
(a) For $n=2$ players: ordinal EF1 after each day, as well as in each individual day;
(b) For $n\geq 3$ players: ordinal EF1 at each individual day, as well as ordinal PROP1 after the last day.
They also prove that, for $n\geq 3$, ordinal EF1 after each day cannot be guaranteed even when all players have identical preferences.

In contrast to these works, we focus on the important special case of one item per player per day, and consider ordinal fairness guarantees that should hold after each day, for any $n$.

\paragraph{Temporal fair allocation.}
This setting is even more general, as it allows the number and values of items to be different in each day. 
The guarantees are naturally weaker; see \citet{he2019achieving,elkind2024temporal,cookson2025temporal}. 
Still, the fact that the same players participate in each daily allocation is advantageous, as it allows us to compensate players who got a low value in early days, by giving them a higher value in later days.

\paragraph{One-shot allocation.}
The most well-studied fair division setting involves a one-shot allocation (a single day). For this setting, there are several algorithms that attain an allocation that is \emph{envy-free up to one item (EF1)} --- a condition stronger than PROP1 \citep{lipton2004approximately,budish2011combinatorial,caragiannis2019unreasonable}. These cannot be applied in our setting, as they might allocate more than one item to the same player in the same day.

\paragraph{One-shot assignment.} 
Here, each player must receive a single item. In this setting every allocation is trivially PROP1 (and EF1), which makes these conditions uninteresting. Fairness is usually attained either by randomization or by monetary transfers; see \citet{basteck2018fair} and \citet{segal2022generalized} respectively for some recent surveys.
In our \emph{repeated} assignment setting, we expect to attain fair outcomes even without randomization or money.

\paragraph{One-shot allocation with cardinality constraints.}
Here, items are partitioned into pre-specified categories, and each player can receive at most a fixed number of items (called the ``capacity'') of each category
 \citep{biswas2018fair,dror2023fair,shoshan2023efficient}.
 The repeated assignment setting can be reduced to this setting: define, for each day, a category containing all items allocated at that day, with a capacity of $1$. There are algorithms that guarantee EF1 (hence also PROP1), but only for specific numeric valuations. For example, the algorithm of \citet{biswas2018fair}, translated to our setting, works as follows: (1) Assign the items in day 1 arbitrarily. (2) In each day $t\geq 2$, order the players by ascending value, from the player with the smallest to the one with the largest total value; give item 1 to the first player in the order, item 2 to the second player, etc. Suppose that, after some day $t$, some player $i$ has a smaller value than some other player $j$, i.e., $v(Z_i^t) <  v(Z_j^t)$. Let $t'$ be the maximum $t'<t$ in which $v(Z_i^{t'}) \geq  v(Z_j^{t'})$ held. Then at day $t'+1$, player $j$ received an item with a higher value than $i$, and at days $t'+2,\ldots,t$, player $i$ received an item with an equal-or-higher value than $j$. Hence, if the item of day $t'+1$ is removed from $j$'s bundle, $i$ does not envy. Hence, the allocation is EF1.
 
 In contrast, our aim is to guarantee \emph{ordinal approximate fairness} --- a condition that holds simultaneously for all numeric valuations consistent with the ranking.
Another advantage of this approach is that our algorithms do not need to know the valuations, as they can allocate positions in a picking sequence.
 
 \paragraph{Picking sequences}
  The output of our algorithms can also be understood as a \emph{picking sequence} --- an order by which each player picks a favorite item from the remaining items. Picking sequences have a very low informational requirements, as players do not have to reveal their valuations --- they only have to pick the best remaining item when their turn arrives \citep{bouveret2011general,brams2004dividing,segalhalevi2020competitive}.

 \paragraph{Online fair allocation.}
 This is the hardest setting studied, as it assumes that the items and their valuations are \emph{not known in advance} (in contrast to temporal fair division). Each day, a new item arrives, and must be allocated immediately and irrevocably to one of the players. As the future is not known, the algorithm might make allocations that are bad in hindsight. This translates to much weaker fairness guarantees. It is not possible to guarantee PROP$c$ for any constant $c$, even in the last round; such a guarantee is possible only if $c$ grows at least like $\sqrt{t}$. See e.g.\citet{benade2018make,benade2024fair,kahana2023leximin,he2019achieving,neoh2025online,zeng2020fairness,jiang2019online}.

\paragraph{Latin squares.}
\citet{kawase2025resource} studied an allocation problem with ``latin square'' constraints. Their goal is not to attain proportionality, but rather to maximize the social welfare.

\begin{credits}
	\subsubsection{\ackname}
	This work has started at the following post on Mathematics Stack Exchange: 
	\url{https://math.stackexchange.com/q/5014281}.
	
	Erel is funded by Israel Science Foundation grant no. 1092/24.
	
	The proof of \Cref{prop:no-weakprop-all-residues} in \Cref{app:no-weak-balanced-sequence} was written by Aristotle \citep{achim2025aristotleimolevelautomatedtheorem} and verified manually by the authors.
	
	\subsubsection{\discintname}
	The authors have no competing interests to declare that are relevant to the content of this article.
\end{credits}

~\\
\let\oldclearpage\clearpage
\let\clearpage\relax
\small
\bibliographystyle{splncs04nat}
\bibliography{main}
\let\clearpage\oldclearpage

\newpage
\appendix
\section{Balanced Sequences}
\label[appendix]{app:fully-balanced-sequences}

\begin{table}
\begin{center}
	\begin{tabular}{cccc}
		\begin{minipage}{.3\linewidth}
\begin{tabular}{|c|c|c|c|c|c|c|}
	\hline
\ifRowIsPlayer
	1 & 7 & 5 & 4 & 2 & 6 & 3 \\
	2 & 6 & 4 & 7 & 3 & 5 & 1 \\
	3 & 5 & 7 & 2 & 6 & 1 & 4 \\ 
	4 & 3 & 6 & 1 & 7 & 2 & 5 \\ 
	5 & 1 & 3 & 6 & 4 & 7 & 2 \\ 
	6 & 4 & 2 & 5 & 1 & 3 & 7 \\ 
	7 & 2 & 1 & 3 & 5 & 4 & 6 \\ 
\else 
	1 & 5 & 7 & 4 & 6 & 3 & 2 \\
	2 & 7 & 6 & 3 & 1 & 4 & 5 \\
	3 & 4 & 5 & 7 & 2 & 6 & 1 \\ 
	4 & 6 & 2 & 1 & 5 & 7 & 3 \\ 
	5 & 3 & 1 & 6 & 7 & 2 & 4 \\ 
	6 & 2 & 4 & 5 & 3 & 1 & 7 \\ 
	7 & 1 & 3 & 2 & 4 & 5 & 6 \\ 
\fi
	\hline
\end{tabular}
		\end{minipage}	
		&  
		\begin{minipage}{.33\linewidth}
\begin{tabular}{|c|c|c|c|c|c|c|c|}
	\hline
\ifRowIsPlayer
	1 & 8 & 6 & 4 & 5 & 3 & 7 & 2 \\
	2 & 7 & 5 & 3 & 8 & 4 & 6 & 1 \\ 
	3 & 6 & 8 & 2 & 4 & 7 & 1 & 5 \\ 
	4 & 5 & 2 & 8 & 7 & 1 & 3 & 6 \\ 
	5 & 4 & 1 & 7 & 3 & 6 & 2 & 8 \\ 
	6 & 3 & 7 & 1 & 2 & 8 & 5 & 4 \\ 
	7 & 2 & 4 & 6 & 1 & 5 & 8 & 3 \\ 
	8 & 1 & 3 & 5 & 6 & 2 & 4 & 7 \\ 
\else
	1 & 8 & 5 & 6 & 7 & 4 & 3 & 2 \\
	2 & 7 & 4 & 3 & 6 & 8 & 5 & 1 \\ 
	3 & 6 & 8 & 2 & 5 & 1 & 4 & 7 \\ 
	4 & 5 & 7 & 1 & 3 & 2 & 8 & 6 \\ 
	5 & 4 & 2 & 8 & 1 & 7 & 6 & 3 \\ 
	6 & 3 & 1 & 7 & 8 & 5 & 2 & 4 \\ 
	7 & 2 & 6 & 5 & 4 & 3 & 1 & 8 \\ 
	8 & 1 & 3 & 4 & 2 & 6 & 7 & 5 \\ 
\fi
	\hline
\end{tabular}
		\end{minipage}	
		&  
		\begin{minipage}{.36\linewidth}
\begin{tabular}{|c|c|c|c|c|c|c|c|c|}
	\hline
\ifRowIsPlayer
	1 & 9 & 6 & 5 & 3 & 8 & 4 & 7 & 2 \\
	2 & 8 & 5 & 7 & 4 & 3 & 9 & 6 & 1 \\ 
	3 & 6 & 9 & 2 & 8 & 4 & 7 & 1 & 5 \\ 
	4 & 7 & 3 & 9 & 2 & 5 & 1 & 8 & 6 \\ 
	5 & 3 & 8 & 1 & 6 & 9 & 2 & 4 & 7 \\ 
	6 & 2 & 7 & 4 & 9 & 1 & 8 & 5 & 3 \\ 
	7 & 5 & 2 & 8 & 1 & 6 & 3 & 9 & 4 \\ 
	8 & 4 & 1 & 6 & 7 & 2 & 5 & 3 & 9 \\ 
	9 & 1 & 4 & 3 & 5 & 7 & 6 & 2 & 8 \\ 
\else
	1 & 9 & 8 & 5 & 7 & 6 & 4 & 3 & 2 \\
	2 & 6 & 7 & 3 & 4 & 8 & 5 & 9 & 1 \\ 
	3 & 5 & 4 & 9 & 1 & 2 & 7 & 8 & 6 \\ 
	4 & 8 & 9 & 6 & 2 & 3 & 1 & 5 & 7 \\ 
	5 & 7 & 2 & 1 & 9 & 4 & 8 & 6 & 3 \\ 
	6 & 3 & 1 & 8 & 5 & 7 & 9 & 2 & 4 \\ 
	7 & 4 & 6 & 2 & 8 & 9 & 3 & 1 & 5 \\ 
	8 & 2 & 5 & 7 & 3 & 1 & 6 & 4 & 9 \\ 
	9 & 1 & 3 & 4 & 6 & 5 & 2 & 7 & 8 \\ 
\fi
	\hline
\end{tabular}
		\end{minipage}	
	\end{tabular}
\end{center}
\caption{
	Fully-balanced permutation sequences for $n\in\{7,8,9\}$.
\label{tab:n=789}
}
\end{table}

\begin{table}
\begin{center}
	\begin{tabular}{cccc}
		\begin{minipage}{.4\linewidth}
\begin{tabular}{|c|c|c|c|c|c|c|c|c|c|}
	\hline
\ifRowIsPlayer
	1 & 10 & 7 & 4 & 6 & 3 & 9 & 2 & 8 & 5\\
	2 & 9 & 6 & 3 & 8 & 5 & 10 & 4 & 1 & 7\\
	3 & 8 & 2 & 6 & 10 & 4 & 7 & 5 & 9 & 1\\ 
	4 & 6 & 10 & 2 & 7 & 1 & 8 & 3 & 5 & 9\\ 
	5 & 7 & 3 & 10 & 2 & 9 & 6 & 1 & 4 & 8\\ 
	6 & 4 & 8 & 1 & 9 & 2 & 5 & 10 & 7 & 3\\ 
	7 & 3 & 9 & 5 & 1 & 10 & 4 & 8 & 6 & 2\\ 
	8 & 5 & 1 & 9 & 4 & 7 & 3 & 6 & 2 & 10\\ 
	9 & 2 & 5 & 8 & 3 & 6 & 1 & 7 & 10 & 4\\ 
	10 & 1 & 4 & 7 & 5 & 8 & 2 & 9 & 3 & 6\\ 
\else
	1 & 10 & 8 & 6 & 7 & 4 & 9 & 5 & 2 & 3\\
	2 & 9 & 3 & 4 & 5 & 6 & 10 & 1 & 8 & 7\\
	3 & 7 & 5 & 2 & 9 & 1 & 8 & 4 & 10 & 6\\ 
	4 & 6 & 10 & 1 & 8 & 3 & 7 & 2 & 5 & 9\\ 
	5 & 8 & 9 & 7 & 10 & 2 & 6 & 3 & 4 & 1\\ 
	6 & 4 & 2 & 3 & 1 & 9 & 5 & 8 & 7 & 10\\ 
	7 & 5 & 1 & 10 & 4 & 8 & 3 & 9 & 6 & 2\\ 
	8 & 3 & 6 & 9 & 2 & 10 & 4 & 7 & 1 & 5\\ 
	9 & 2 & 7 & 8 & 6 & 5 & 1 & 10 & 3 & 4\\ 
	10 & 1 & 4 & 5 & 3 & 7 & 2 & 6 & 9 & 8\\ 
\fi
	\hline
\end{tabular}
		\end{minipage}	
		&  
		\begin{minipage}{.6\linewidth}
\begin{tabular}{|c|c|c|c|c|c|c|c|c|c|c|}
	\hline
\ifRowIsPlayer
	1 & 11 & 8 & 6 & 5 & 4 & 10 & 3 & 9 & 7 & 2\\
	2 & 10 & 7 & 5 & 4 & 11 & 8 & 1 & 6 & 9 & 3\\ 
	3 & 9 & 6 & 11 & 2 & 8 & 5 & 7 & 1 & 4 & 10\\ 
	4 & 8 & 11 & 3 & 7 & 2 & 9 & 6 & 5 & 10 & 1\\ 
	5 & 7 & 4 & 2 & 11 & 10 & 3 & 9 & 8 & 1 & 6\\ 
	6 & 5 & 3 & 10 & 9 & 1 & 7 & 11 & 4 & 2 & 8\\ 
	7 & 4 & 10 & 1 & 8 & 5 & 11 & 2 & 3 & 6 & 9\\ 
	8 & 3 & 9 & 4 & 1 & 6 & 2 & 10 & 7 & 5 & 11\\ 
	9 & 1 & 5 & 7 & 10 & 3 & 6 & 8 & 2 & 11 & 4\\ 
	10 & 6 & 2 & 9 & 3 & 7 & 1 & 4 & 11 & 8 & 5\\ 
	11 & 2 & 1 & 8 & 6 & 9 & 4 & 5 & 10 & 3 & 7\\ 
\else
	1 & 9 & 11 & 7 & 8 & 6 & 10 & 2 & 3 & 5 & 4\\
	2 & 11 & 10 & 5 & 3 & 4 & 8 & 7 & 9 & 6 & 1\\ 
	3 & 8 & 6 & 4 & 10 & 9 & 5 & 1 & 7 & 11 & 2\\ 
	4 & 7 & 5 & 8 & 2 & 1 & 11 & 10 & 6 & 3 & 9\\ 
	5 & 6 & 9 & 2 & 1 & 7 & 3 & 11 & 4 & 8 & 10\\ 
	6 & 10 & 3 & 1 & 11 & 8 & 9 & 4 & 2 & 7 & 5\\ 
	7 & 5 & 2 & 9 & 4 & 10 & 6 & 3 & 8 & 1 & 11\\ 
	8 & 4 & 1 & 11 & 7 & 3 & 2 & 9 & 5 & 10 & 6\\ 
	9 & 3 & 8 & 10 & 6 & 11 & 4 & 5 & 1 & 2 & 7\\ 
	10 & 2 & 7 & 6 & 9 & 5 & 1 & 8 & 11 & 4 & 3\\ 
	11 & 1 & 4 & 3 & 5 & 2 & 7 & 6 & 10 & 9 & 8\\ 
\fi
	\hline
\end{tabular}
		\end{minipage}	
	\end{tabular}
\end{center}
\caption{
	Fully-balanced permutation sequences for $n\in\{10,11\}$.
\label{tab:n=1011}
}
\end{table}

\begin{table}[h!]
	\centering 
	\begin{tabular}{|c|c|c|c|c|c|c|c|c|c|c|c|}
		\hline 
\ifRowIsPlayer
		1 & 12 & 9 & 7 & 5 & 11 & 4 & 8 & 3 & 10 & 6 & 2\\ 
		2 & 11 & 8 & 6 & 4 & 12 & 7 & 5 & 1 & 9 & 10 & 3\\ 
		3 & 10 & 7 & 5 & 12 & 9 & 2 & 6 & 11 & 4 & 8 & 1\\ 
		4 & 9 & 12 & 3 & 8 & 7 & 1 & 11 & 6 & 5 & 2 & 10\\ 
		5 & 8 & 11 & 4 & 3 & 2 & 12 & 10 & 7 & 6 & 1 & 9\\ 
		6 & 7 & 5 & 2 & 11 & 10 & 3 & 12 & 9 & 1 & 4 & 8\\ 
		7 & 6 & 4 & 12 & 2 & 8 & 11 & 1 & 10 & 3 & 9 & 5\\ 
		8 & 5 & 10 & 1 & 9 & 4 & 6 & 3 & 12 & 2 & 7 & 11\\ 
		9 & 4 & 6 & 11 & 1 & 5 & 10 & 2 & 8 & 12 & 3 & 7\\ 
		10 & 3 & 2 & 9 & 7 & 1 & 8 & 4 & 5 & 11 & 12 & 6\\ 
		11 & 2 & 1 & 10 & 6 & 3 & 9 & 7 & 4 & 8 & 5 & 12\\ 
		12 & 1 & 3 & 8 & 10 & 6 & 5 & 9 & 2 & 7 & 11 & 4\\ 
\else
		1 & 12 & 11 & 8 & 9 & 10 & 4 & 7 & 2 & 6 & 5 & 3\\ 
		2 & 11 & 10 & 6 & 7 & 5 & 3 & 9 & 12 & 8 & 4 & 1\\ 
		3 & 10 & 12 & 4 & 5 & 11 & 6 & 8 & 1 & 7 & 9 & 2\\ 
		4 & 9 & 7 & 5 & 2 & 8 & 1 & 10 & 11 & 3 & 6 & 12\\ 
		5 & 8 & 6 & 3 & 1 & 9 & 12 & 2 & 10 & 4 & 11 & 7\\ 
		6 & 7 & 9 & 2 & 11 & 12 & 8 & 3 & 4 & 5 & 1 & 10\\ 
		7 & 6 & 3 & 1 & 10 & 4 & 2 & 11 & 5 & 12 & 8 & 9\\ 
		8 & 5 & 2 & 12 & 4 & 7 & 10 & 1 & 9 & 11 & 3 & 6\\ 
		9 & 4 & 1 & 10 & 8 & 3 & 11 & 12 & 6 & 2 & 7 & 5\\ 
		10 & 3 & 8 & 11 & 12 & 6 & 9 & 5 & 7 & 1 & 2 & 4\\ 
		11 & 2 & 5 & 9 & 6 & 1 & 7 & 4 & 3 & 10 & 12 & 8\\ 
		12 & 1 & 4 & 7 & 3 & 2 & 5 & 6 & 8 & 9 & 10 & 11\\ 
\fi
		\hline
	\end{tabular}
~\\
~\\
    \caption{weakly-fully-balanced sequence for $n=12$}\label{table-wprop-12}
\end{table}

\newpage
\section{Integers for which No Fully-Balanced Sequence Exists}
\label[appendix]{app:no-balanced-sequence}

We prove the following extension of \Cref{lem:n=6k}.

\begin{proposition}
	There is no fully-balanced permutation sequence of length $n$ in any 	of the following ranges:
	\[
	\begin{array}{c|c|c}
		n\bmod 6&n&\text{first value covered}\tabularnewline \hline
		0&n=6k,\ k\ge   2&12\\
		1&n=6k+1,\ k\ge 11&67\\
		2&n=6k+2,\ k\ge 3&20\\
		3&n=6k+3,\ k\ge 8&51\\
		4&n=6k+4,\ k\ge 4&28\\
		5&n=6k+5,\ k\ge 9&59.
	\end{array}
	\]
	In particular, no fully-balanced sequence exists for any $n > 61$.
\end{proposition}

The technique used here is similar to that in the proof of \Cref{lem:n=6k}. We also use the same notation $A_i, B_i, C_i, D_i$ as in that proof.

\subsection{Case $n=6k+1$}
The following formulas hold:
$$\Big\lceil \frac{6k+1}{2}\Big\rceil=3k+1,\ \ \Big\lceil \frac{6k+1}{3}\Big\rceil=2k+1\ \ $$
and 
\begin{eqnarray*} 
\Big\lceil \frac{6k+1}{4}\Big\rceil =& 
\left\{\begin{array}{ll}
{(6k+4)}/{4} & \mbox{if $k$ is even} \\ 
{(6k+2)}/{4} & \mbox{if $k$ is odd}
\end{array}
\right.
\end{eqnarray*}
We have the following: 
$\{A_{3k+2},\ldots,A_{6k+1}\}\subset \{B_{1},\ldots,B_{3k+1}\}$ and 
$\{B_{3k+2},\ldots,B_{6k+1}\}\subset \{A_{1},\ldots,A_{3k+1}\}$. 
Let 
$$\mathcal{A}_k=\{A_{3k+2},\ldots,A_{6k+1}\}\cap \{B_{2k+2},\ldots,B_{3k+1}\}$$
and 
$$\mathcal{B}_k=\{B_{3k+2},\ldots,B_{6k+1}\}\cap \{A_{2k+2},\ldots,A_{3k+1}\}.$$
Thus, $\mathcal{A}_k$ and $\mathcal{B}_k$ are disjoint and 
$|\mathcal{A}_k|\geq k-1$, $|\mathcal{B}_k|\geq k-1$. 

\subsubsection{Even $k$}
Let $a_k=\frac{6k+4}{4}$ and $\mathcal{C}_k=\{C_{a_k+1},\ldots,C_{2k+1}\}$. 
Since $\mathcal{A}_k\cup \mathcal{B}_k\subset \{C_{1},\ldots,C_{2k+1}\}$, 
then 
$$\Big|(\mathcal{A}_k\cup \mathcal{B}_k)\cap \mathcal{C}_k\Big|
\geq 2k+1-a_k-3=\frac{1}{2}k-3.$$
We know that 
$(\mathcal{A}_k\cup \mathcal{B}_k)\cap \mathcal{C}_k\subset \{D_{1},\ldots,D_{a_k}\}$.  
Also, 
$$\Big|\{C_{3k+2},\ldots,C_{6k+1}\}\cap \{D_{1},\ldots,D_{3k+1}\}\Big|
\geq 3k-1.$$
There is a contradiction, if 
$3k-1+\frac{1}{2}k-3 > 3k+1$.  Hence, there is no solution, if $k>10$, or 
equivalently, $k\geq 12$ and $n\geq 73$. 

\subsubsection{Odd $k$}
Using an argument similar to even $k$, we get a contradiction, if 
$$3k-1+\frac{1}{2}k-\frac{5}{2} > 3k+1.$$
Hence, there are no solutions if $k\geq 11$ or $n\geq 67$. 

\subsection{Case $n=6k+2$}
The following formulas hold:
$$\Big\lceil \frac{6k+2}{2}\Big\rceil=3k+1,\ \ \Big\lceil \frac{6k+2}{3}\Big\rceil=2k+1\ \ $$
and 
\begin{eqnarray*} 
\Big\lceil \frac{6k+2}{4}\Big\rceil =& 
\left\{\begin{array}{ll}
{(6k+4)}/{4} & \mbox{if $k$ is even} \\ 
{(6k+2)}/{4} & \mbox{if $k$ is odd}
\end{array}
\right.
\end{eqnarray*}
We have the following: 
$\{A_{3k+2},\ldots,A_{6k+2}\} =\{B_{1},\ldots,B_{3k+1}\}$ and 
$\{B_{3k+2},\ldots,B_{6k+2}\} =\{A_{1},\ldots,A_{3k+1}\}$. 
Let 
$$\mathcal{A}_k=\{A_{3k+2},\ldots,A_{6k+2}\}\cap \{B_{2k+2},\ldots,B_{3k+1}\}$$
and 
$$\mathcal{B}_k=\{B_{3k+2},\ldots,B_{6k+2}\}\cap \{A_{2k+2},\ldots,A_{3k+1}\}.$$
Thus, $\mathcal{A}_k$ and $\mathcal{B}_k$ are disjoint and 
$|\mathcal{A}_k|=k$, $|\mathcal{B}_k|=k$. 

\subsubsection{Even $k$}
Let $a_k=\frac{6k+4}{4}$ and $\mathcal{C}_k=\{C_{a_k+1},\ldots,C_{2k+1}\}$. 
Since $\mathcal{A}_k\cup \mathcal{B}_k\subset \{C_{1},\ldots,C_{2k+1}\}$, 
then 
$$\Big|(\mathcal{A}_k\cup \mathcal{B}_k)\cap \mathcal{C}_k\Big|
\geq 2k+1-a_k-1=\frac{1}{2}k-1.$$
We know that 
$(\mathcal{A}_k\cup \mathcal{B}_k)\cap \mathcal{C}_k\subset \{D_{1},\ldots,D_{a_k}\}$.  
However, 
$$\{C_{3k+2},\ldots,C_{6k+2}\} =\{D_{1},\ldots,D_{3k+1}\}.$$
There is a contradiction, if 
$\frac{1}{2}k-1>0$.  Hence, there is no solution, if $k>2$, or 
equivalently, $n\geq 26$. 

\subsubsection{Odd $k$}
Using an argument similar to even $k$, we get a contradiction, if 
$$\frac{1}{2}k-\frac{1}{2} > 0.$$
Hence, there are no solutions if $k>1$ or $n\geq 20$. 

\subsection{Case $n=6k+3$}
The following formulas hold:
$$\Big\lceil \frac{6k+3}{2}\Big\rceil=3k+2,\ \ \Big\lceil \frac{6k+3}{3}\Big\rceil=2k+1\ \ $$
and 
\begin{eqnarray*} 
\Big\lceil \frac{6k+3}{4}\Big\rceil =& 
\left\{\begin{array}{ll}
{(6k+4)}/{4} & \mbox{if $k$ is even} \\ 
{(6k+6)}/{4} & \mbox{if $k$ is odd}
\end{array}
\right.
\end{eqnarray*}
We have the following: 
$\{A_{3k+3},\ldots,A_{6k+3}\} \subset \{B_{1},\ldots,B_{3k+2}\}$ and 
$\{B_{3k+3},\ldots,B_{6k+3}\} \subset \{A_{1},\ldots,A_{3k+2}\}$. 
Let 
$$\mathcal{A}_k=\{A_{3k+3},\ldots,A_{6k+3}\}\cap \{B_{2k+2},\ldots,B_{3k+2}\}$$
and 
$$\mathcal{B}_k=\{B_{3k+3},\ldots,B_{6k+3}\}\cap \{A_{2k+2},\ldots,A_{3k+2}\}.$$
Thus, $\mathcal{A}_k$ and $\mathcal{B}_k$ are disjoint and 
$|\mathcal{A}_k|=k$, $|\mathcal{B}_k|=k$. 

\subsubsection{Even $k$}
Let $a_k=\frac{6k+4}{4}$ and $\mathcal{C}_k=\{C_{a_k+1},\ldots,C_{2k+1}\}$. 
Since $\mathcal{A}_k\cup \mathcal{B}_k\subset \{C_{1},\ldots,C_{2k+1}\}$, 
then 
$$\Big|(\mathcal{A}_k\cup \mathcal{B}_k)\cap \mathcal{C}_k\Big|
\geq 2k+1-a_k-1=\frac{1}{2}k-1.$$
We know that 
$(\mathcal{A}_k\cup \mathcal{B}_k)\cap \mathcal{C}_k\subset 
\{D_{1},\ldots,D_{a_k}\}$.  
However, 
$$\Big|\{C_{3k+3},\ldots,C_{6k+3}\}\cap \{D_{1},\ldots,D_{3k+2}\}\Big|
\geq 3k.$$

There is a contradiction, if 
$\frac{1}{2}k-1>2$.  Hence, there is no solution, if $k>6$, or 
equivalently, $n\geq 51$. 

\subsubsection{Odd $k$}
Using an argument similar to even $k$, we get a contradiction, if 
$$\frac{1}{2}k-\frac{3}{2} > 2.$$
Hence, there are no solutions if $k>7$ or $n\geq 57$. 

\subsection{Case $n=6k+4$}
The following formulas hold:
$$\Big\lceil \frac{6k+4}{2}\Big\rceil=3k+2,\ \ \Big\lceil \frac{6k+4}{3}\Big\rceil=2k+2\ \ $$
and 
\begin{eqnarray*} 
\Big\lceil \frac{6k+4}{4}\Big\rceil =& 
\left\{\begin{array}{ll}
{(6k+4)}/{4} & \mbox{if $k$ is even} \\ 
{(6k+6)}/{4} & \mbox{if $k$ is odd}
\end{array}
\right.
\end{eqnarray*}
We have the following: 
$\{A_{3k+3},\ldots,A_{6k+4}\} =\{B_{1},\ldots,B_{3k+2}\}$ and 
$\{B_{3k+3},\ldots,B_{6k+4}\} =\{A_{1},\ldots,A_{3k+2}\}$. 
Let 
$$\mathcal{A}_k=\{A_{3k+3},\ldots,A_{6k+4}\}\cap \{B_{2k+3},\ldots,B_{3k+2}\}$$
and 
$$\mathcal{B}_k=\{B_{3k+3},\ldots,B_{6k+4}\}\cap \{A_{2k+3},\ldots,A_{3k+2}\}.$$
Thus, $\mathcal{A}_k$ and $\mathcal{B}_k$ are disjoint and 
$|\mathcal{A}_k|=k$, $|\mathcal{B}_k|=k$. 

\subsubsection{Even $k$}
Let $a_k=\frac{6k+4}{4}$ and $\mathcal{C}_k=\{C_{a_k+1},\ldots,C_{2k+2}\}$. 
Since $\mathcal{A}_k\cup \mathcal{B}_k\subset \{C_{1},\ldots,C_{2k+2}\}$, 
then 
$$\Big|(\mathcal{A}_k\cup \mathcal{B}_k)\cap \mathcal{C}_k\Big|
\geq 2k+2-a_k-2=\frac{1}{2}k-1.$$
We know that 
$(\mathcal{A}_k\cup \mathcal{B}_k)\cap \mathcal{C}_k\subset \{D_{1},\ldots,D_{a_k}\}$.  
However, 
$$\{C_{3k+3},\ldots,C_{6k+4}\} =\{D_{1},\ldots,D_{3k+2}\}.$$
There is a contradiction, if 
$\frac{1}{2}k-1>0$.  Hence, there is no solution, if $k>2$, or 
equivalently, $n\geq 28$. 

\subsubsection{Odd $k$}
Using an argument similar to even $k$, we get a contradiction, if 
$$\frac{1}{2}k-\frac{3}{2} > 0.$$
Hence, there are no solutions if $k>3$ or $n\geq 34$. 

\subsection{Case $n=6k+5$}
The following formulas hold:
$$\Big\lceil \frac{6k+5}{2}\Big\rceil=3k+3,\ \ \Big\lceil \frac{6k+5}{3}\Big\rceil=2k+2\ \ $$
and 
\begin{eqnarray*} 
\Big\lceil \frac{6k+5}{4}\Big\rceil =& 
\left\{\begin{array}{ll}
{(6k+8)}/{4} & \mbox{if $k$ is even} \\ 
{(6k+6)}/{4} & \mbox{if $k$ is odd}
\end{array}
\right.
\end{eqnarray*}
We have the following: 
$\{A_{3k+4},\ldots,A_{6k+5}\} \subset \{B_{1},\ldots,B_{3k+3}\}$ and 
$\{B_{3k+4},\ldots,B_{6k+5}\} \subset \{A_{1},\ldots,A_{3k+3}\}$. 
Let 
$$\mathcal{A}_k=\{A_{3k+4},\ldots,A_{6k+5}\}\cap \{B_{2k+3},\ldots,B_{3k+3}\}$$
and 
$$\mathcal{B}_k=\{B_{3k+4},\ldots,B_{6k+5}\}\cap \{A_{2k+3},\ldots,A_{3k+3}\}.$$
Thus, $\mathcal{A}_k$ and $\mathcal{B}_k$ are disjoint and 
$|\mathcal{A}_k|=k$, $|\mathcal{B}_k|=k$. 

\subsubsection{Even $k$}
Let $a_k=\frac{6k+8}{4}$ and $\mathcal{C}_k=\{C_{a_k+1},\ldots,C_{2k+2}\}$. 
Since $\mathcal{A}_k\cup \mathcal{B}_k\subset \{C_{1},\ldots,C_{2k+2}\}$, 
then 
$$\Big|(\mathcal{A}_k\cup \mathcal{B}_k)\cap \mathcal{C}_k\Big|
\geq 2k+2-a_k-2=\frac{1}{2}k-2.$$
We know that 
$(\mathcal{A}_k\cup \mathcal{B}_k)\cap \mathcal{C}_k\subset 
\{D_{1},\ldots,D_{a_k}\}$.  
However, 
$$\Big|\{C_{3k+4},\ldots,C_{6k+5}\}\cap \{D_{1},\ldots,D_{3k+3}\}\Big|
\geq 3k+1.$$
There is a contradiction, if 
$\frac{1}{2}k-2>2$.  Hence, there is no solution, if $k>8$, or 
equivalently, $n\geq 65$. 

\subsubsection{Odd $k$}
Using an argument similar to even $k$, we get a contradiction, if 
$$\frac{1}{2}k-\frac{3}{2} > 2.$$
Hence, there are no solutions if $k>7$ or $n\geq 59$.

\section{Integers for which No Weakly Fully-Balanced Sequence Exists}
\label[appendix]{app:no-weak-balanced-sequence}
\begin{proposition}
\label{prop:no-weakprop-all-residues}
There is no weakly-fully-balanced permutation sequence of length $n$ in any
of the following ranges:
\[
\begin{array}{c|c|c}
	n\bmod 6&n&\text{first value covered}\tabularnewline \hline
	0&n=6k,\ k\ge19&114\\
	1&n=6k+1,\ k\ge11&67\\
	2&n=6k+2,\ k\ge16&98\\
	3&n=6k+3,\ k\ge12&75\\
	4&n=6k+4,\ k\ge17&106\\
	5&n=6k+5,\ k\ge9&59.
\end{array}
\]
In particular, no weakly-fully-balanced sequence exists for any $n > 108$.
\end{proposition}

\begin{proof}
Put
\[
a=\left\lfloor\frac n2\right\rfloor+1,\qquad
b=\left\lfloor\frac n3\right\rfloor+1,\qquad
d=\left\lfloor\frac n4\right\rfloor+1,\qquad
e=2a-n.
\]
The number $e$ is $2$ for even $n$ and $1$ for odd $n$.
By $\weakpropcond(2,1)$, every player receiving an item ranked above $a$
in day 1 receives an item of rank at most $a$ in day 2.  Consequently,
at most $e$ players receive top-$a$ items in both days 1 and 2: the
$2a$ top-item incidences are distributed among $n$ players.

Using the notation $A_i, B_i, C_i, D_i$ as in the proof of \Cref{lem:n=6k}, define:
\begin{align*}
Z_1=\rangeA{b+1}{a}\cap\rangeB{a+1}{n},
\\
Z_2=\rangeB{b+1}{a}\cap\rangeA{a+1}{n}.
\end{align*}
At most $e$ players in $\rangeA{b+1}{a}$ can have received a top-$a$ item in day 2. So at most $e$ players 
in $\rangeA{b+1}{a}$ are not in $\rangeB{a+1}{n}$. 
Hence, $|Z_1|\geq a-b-e$. Analogous arguments imply $|Z_2|\geq a-b-e$.

The sets $Z_1,Z_2$ are disjoint by construction.  By $\weakpropcond(3,1)$, all their players must
receive a top-$b$ item in day 3, so
\[
Z_1\cup Z_2\subseteq\rangeC{1}{b}.
\]
After discarding the $d$ players receiving a top-$d$ item in day 3,
at least
\[
2(a-b-e)-d
\]
of these players must receive a top-$d$ item in day 4, by
$\weakpropcond(4,1)$.
Independently, consider the $n-a$ players in $\rangeC{a+1}{n}$.  At most
$e$ of them received top-$a$ items in both A and B.  Every other such
player has received a top-$a$ item at most once in the first three days,
and hence must receive a top-$a$ item in D by
$\weakpropcond(4,2)$.  This contributes at least $n-a-e$ players to
$\rangeD{1}{a}$.  They are disjoint from the former contribution, since
their C-ranks are respectively above $a$ and between $d+1$ and $b$.
Therefore a necessary condition is
\[
a\ge \bigl(2(a-b-e)-d\bigr)+(n-a-e),
\]
or equivalently
\begin{align}
	\label{eq:weak-balance-general-necessary}
	n&\le 2b+d+3e.
\end{align}
It remains to check what values of $n$ satisfy \eqref{eq:weak-balance-general-necessary}. Here we split by $n$ mod $6$:
\begin{itemize}
\item For $n=6k$ and $k$ even, we have $b=2k+1, d=1.5k+1, e=2$, so \eqref{eq:weak-balance-general-necessary} becomes $6k \leq 4k + 2 + 1.5k + 1 + 6 = 5.5k + 9$, equivalently $k\leq 18$.

If $k$ is odd, then $d$ becomes $1.5k+0.5$, so \eqref{eq:weak-balance-general-necessary} becomes $6k \leq 5.5k + 8.5$, equivalently $k\leq 17$.

\item For $n=6k+1$ and $k$ even, we have $b=2k+1, d=1.5k+1, e=1$, so \eqref{eq:weak-balance-general-necessary} becomes $6k+1 \leq 4k + 2 + 1.5k + 1 + 3 = 5.5k + 6$, equivalently $k\leq 10$.

If $k$ is odd, then \eqref{eq:weak-balance-general-necessary} becomes $6k+1\leq 5.5k + 5.5$, equivalently $k\leq 9$.
	
\item For $n=6k+2$ and $k$ even, we have $b=2k+1, d=1.5k+1, e=2$, so \eqref{eq:weak-balance-general-necessary} becomes $6k+2 \leq 4k + 2 + 1.5k + 1 + 6 = 5.5k + 9$, equivalently $k\leq 14$.

If $k$ is odd, then $d$ becomes $1.5k+1.5$, so \eqref{eq:weak-balance-general-necessary} becomes $6k+2 \leq 5.5k + 9.5$, equivalently $k\leq 15$.

\item For $n=6k+3$ and $k$ even, we have $b=2k+2, d=1.5k+1, e=1$, so \eqref{eq:weak-balance-general-necessary} becomes $6k+3 \leq 4k + 4 + 1.5k + 1 + 3 = 5.5k + 8$, equivalently $k\leq 10$.

If $k$ is odd, then $d$ becomes $1.5k+1.5$, so \eqref{eq:weak-balance-general-necessary} becomes $6k+3 \leq 5.5k + 8.5$, equivalently $k\leq 11$.

\item For $n=6k+4$ and $k$ even, we have $b=2k+2, d=1.5k+2, e=2$, so \eqref{eq:weak-balance-general-necessary} becomes $6k+4 \leq 4k + 4 + 1.5k + 2 + 6 = 5.5k + 12$, equivalently $k\leq 16$.

If $k$ is odd, then $d$ becomes $1.5k+1.5$, so \eqref{eq:weak-balance-general-necessary} becomes $6k+4 \leq 5.5k + 11.5$, equivalently $k\leq 15$.

\item For $n=6k+5$ and $k$ even, we have $b=2k+2, d=1.5k+2, e=1$, so \eqref{eq:weak-balance-general-necessary} becomes $6k+5 \leq 4k + 4 + 1.5k + 2 + 3 = 5.5k + 9$, equivalently $k\leq 8$.

If $k$ is odd, then $d$ becomes $1.5k+1.5$, so \eqref{eq:weak-balance-general-necessary} becomes $6k+5 \leq 5.5k + 8.5$, equivalently $k\leq 7$.
\end{itemize}
These thresholds correspond to the thresholds in the proposition statement.
\end{proof}

\section{Example of Top-Balanced Sequence Algorithm}
\label[appendix]{app:top-balanced-sequence}
Here, we demonstrate a top-balanced sequence for $n=15$ players.  
This example is sufficiently complex to show how 
this procedure works for any positive integer $n$.  

For day $1$, assign items $1$ through $15$ to players 
$1$ through $15$ in increasing order. 
For day $2$, assign items $1$ through $7$ to players 
$9$ through $15$.  The remaining items can be assigned 
arbitrarily to the remaining players for day $2$.  
For day $t=3$, we have $\lceil \frac{15}{3}\rceil = 5$.  
Notice that for the first $2$ days, there are $10$ unique players 
who were assigned one of the items $1$ through $5$.  
There are $5$ players not in this 
group.  In particular, these are players $6,7,8,14,15$.  Assign 
these players items $1$ through $5$.  Assign the remaining items 
arbitrarily to the remaining players for day $3$.  

For day $t=4$, we have $\lceil \frac{15}{4}\rceil = 4$.  
For the first $3$ days, the players assigned one of the items 
$1$ through $4$ are $12$ unique players.  For day $4$, choose 
the $3$ players who are not in this group of $12$ players, 
and assign items $1$ through $3$ to these players.  
These will be players $5,13,15$.  
Assign the remaining items arbitrarily for day $4$.  

For day $t=5$, we have $\lceil \frac{15}{5}\rceil = 3$.  
A key point is that the group of players assigned one of 
items $1$ through $3$ on the first $4$ days are unique.  
There are $3$ players outside this group.  
Given the previous assignments, these will be
players $4,12,14$.  Assign items 
$1$ through $3$ to these $3$ players on day $5$.  
Assign the remaining items to the remaining players 
arbitrarily. 

For day $t=6$, we have $\lceil \frac{15}{6}\rceil = 3$.  
Every player has already been assigned one of items $1$ through $3$.  
However, we still need some care in assigning items on day $6$.  
Choose $2$ players that were not assigned items $1$ or $2$, 
and assign them items $1$ and $2$ on day $6$.  
These could be players $3$ and $11$.  Assign the 
remaining items arbitrarily.  

For day $t=7$, we still have $\lceil \frac{15}{7}\rceil = 3$.  
There are at least $2$ players who were not assigned items 
$1$ or $2$ on the first $6$ days.  Assign items $1$ and $2$ 
to these players on day $7$.  These could be players 
$8$ and $15$.  

On day $t=8$, there is one player not previously 
assigned items $1$ or $2$.  Based on previous assignments, 
this would be player $14$.  
Assign item $1$ to this player.  The remaining items may be 
assigned arbitrarily on day $8$.

Every player has been assigned either item $1$ or item $2$ 
after $8$ days.  In particular, $8$ different players were 
assigned item $1$.  Arbitrarily order the remaining $7$ 
players and assign item $1$ to each of the $7$ players 
on each of the remaining days.  
The remaining $7$ players could be ordered as 
$$2,10,7,13,12,8,14.$$

This procedure can be applied to an arbitrary positive integer $n$.  
The key property is that at each stage 
all players are assigned one of the top $\lceil \frac{n}{t}\rceil$ 
items in the first $t$ days.  If $\frac{n}{t}$ is an integer, 
then these players will be unique.  Moving to the next day, 
we pick the players that have not been assigned any items from 
$1$ through $\lceil \frac{n}{t+1}\rceil$ and assign those items 
to these players on day $t+1$.  
As soon as those players have been exhausted, 
we focus on the players that have not been assigned any items 
from $1$ through $\lceil \frac{n}{t+2}\rceil$ and so on.  
This procedure will produce a top-balanced sequence. 

\newpage

\section{Heterogeneous valuations}
\label[appendix]{app:het-valuations}
Throughout the paper we assumed for simplicity that all players have the same ranking and the same value function. Here we prove that our main positive result, namely  \Cref{thm:weak-full-balance-implies-ordinal-prop1}, holds even when players have different valuations and different rankings.

\begin{theorem}
	\label{thm:weak-full-balance-implies-ordinal-prop1-het}
	A weakly fully-balanced permutation sequence generates a perpetually ordinally PROP1 repeated assignment,
	even when players have different rankings and different valuations.
\end{theorem}
\begin{proof}
As different players may rank the items differently, we can no longer identify a number $j$ with ``the $j$-th highest ranked item''. Instead, we interpret the entries of the permutation sequence as \emph{picking rights}: the number $j$ received by a player on a given day is \emph{the right to be the $j$-th one to pick an item} on that day. That is, on each day the players pick items one at a time, in increasing order of their number, so that the player holding right $1$ picks first, the player holding right $2$ picks second, and so on. A player holding right $j$ always gets an item that is \emph{at least as good, from that player's own viewpoint}, as their $j$-th best item. Hence, if the player's valuation is $v_i$, the value they obtain from a single pick with right $j$ is at least $v_i(j)$, where we rename the items so that $v_i(1)\geq \cdots \geq v_i(n)$.
	
	Crucially, the weak-balance condition constrains the picking rights only, and is independent of any player's valuation. For a player $i$ and a day $t$, let $P_i^t$ denote the multiset of picking rights that $i$ has accumulated over days $1,\ldots,t$, and let $P_i^t[j]$ denote the $j$-th smallest (that is, $j$-th best) right in $P_i^t$. In this notation, $\weakpropcond(t,j)$ states that
	\[
	P_i^t[j] \leq \floor{j n / t + 1}
	\qquad \text{for all } j\in[t].
	\]
	
	We now fix a day $t$ and a player $i$ with valuation $v_i$, rename the items as above 
	so that $v_i(1)\geq \cdots \geq v_i(n)$, and prove that the cumulative allocation up to day $t$ satisfies PROP1 for $i$.
	
	So far, $n t$ picking rights have been allocated in total, so $|Z_N^t| = n t$, where $Z_N^t$ now denotes the multiset of all allocated rights. Since each day is a permutation of the rights $1,\ldots,n$, the multiset $Z_N^t$ contains exactly $t$ clones of every right in $1,\ldots,n$. Because the actual item allocated for right $j$ is worth at least $v_i(j)$ to player $i$, it suffices to bound $i$'s value from below by the value of the rights: writing $v_i(P_i^t) := \sum_{p\in P_i^t} v_i(p)$, the true value of $i$'s bundle is at least $v_i(P_i^t)$, while the true value of all allocated items equals $v_i(Z_N^t) = t\sum_{k=1}^{n} v_i(k)$, which is exactly the value obtained by reading $Z_N^t$ as a multiset of rights. Hence it is enough to establish PROP1 at the level of rights.
	
	Let us add to $P_i^t$ a copy of the best right ($1$) and remove its worst (highest-numbered) right, so that the modified multiset still has $t$ rights overall. Denote this new multiset $C_i$. Adding right $1$ corresponds to adding a single most-preferred item to $i$'s bundle, and removing the worst right corresponds to removing a single item; this is a valid PROP1 add-one/remove-one modification. To show PROP1, it is therefore sufficient to show that $v_i(C_i)\geq v_i(Z_N^t)/n$.
	
	Construct a new multiset $D_i$, containing $n$ clones of each right in $C_i$; now it is sufficient to show that $v_i(D_i)\geq v_i(Z_N^t)$. We do this by exhibiting a bijection $f: D_i \to Z_N^t$ such that $x \leq f(x)$ for all $x\in D_i$; since $v_i$ is non-increasing in the right index, this gives $v_i(x)\geq v_i(f(x))$ for all $x\in D_i$, and summing yields the claim.
	
	We first show the bijection for the special case that $n/t$ is an integer.
	\begin{itemize}
		\item Map the $n$ clones of right $1$ in $D_i$ (the added right) to the $t$ clones of rights $1,2,\ldots,n/t$ in $Z_N^t$.
		
		\item Next, map the $n$ clones of the best right held by $i$, namely $P_i^t[1]$. By $\weakpropcond(t,j=1)$, $P_i^t[1]\leq \floor{n/t+1} = n/t + 1$. Hence we can map these clones to the $t$ clones of rights $n/t+1,\, n/t+2,\ldots,2 n/t$ in $Z_N^t$.
		
		\item Similarly, by $\weakpropcond(t,j)$ for all $j\in[t-1]$,
		$P_i^t[j]\leq j n /t + 1$; hence we can map the $n$ clones of $P_i^t[j]$ to the $t$ clones of rights $j n/t+1,\, j n/t+2,\ldots,(j+1) n/t$ in $Z_N^t$.
	\end{itemize}
	
	When $n/t$ is not an integer, the mapping works as follows:
	
	\begin{itemize}
		\item The $n$ clones of right $1$ in $D_i$ are mapped to the $t$ clones of rights $1,2,\ldots,\floor{n/t}$ in $Z_N^t$, as well as to some $n \% t$ clones of right $\floor{n/t+1}$ (here $\%$ denotes the integer remainder operation).
		
		\item The $n$ clones of $P_i^t[1]$ in $D_i$, which is at most $\floor{n/t+1}$ by $\weakpropcond(t,j=1)$, are mapped to the remaining $t-n \% t$ clones of right $\floor{n/t+1}$ in $Z_N^t$, as well as to the $t$ clones of $\floor{n/t+2}$ and some lower-ranked rights.
		Overall, the $2n$ clones of $1$ and $P_i^t[1]$ in $D_i$ are mapped to the $t$ clones of rights $1,\ldots,\floor{2n/t}$ as well as to some clones of $\floor{2n/t+1}$ in $Z_N^t$.
		
		\item Similarly, for all $j\in[t-1]$, the $(j-1)n$ clones of $1, P_i^t[1],\ldots,P_i^t[j-1]$ in $D_i$ are mapped to the $t$ clones of rights $1,\ldots, \floor{j n/t}$, as well as to some clones of $\floor{j n / t+1}$ in $Z_N^t$.
		By $\weakpropcond(t,j)$, $P_i^t[j]\leq \floor{j n /t + 1}$, so the clones of $P_i^t[j]$ can be mapped to the remaining clones of $\floor{j n /t + 1}$ and to lower-ranked rights.
	\end{itemize}
	
	\begin{table}
		\begin{center}
			\begin{tabular}{cc}
				\begin{minipage}{.5\linewidth}
					\begin{tabular}{|p{5mm}|p{5mm}|p{5mm}|p{5mm}|p{5mm}|p{5mm}|}
						\hline
						1 & 1 & 1 & 1 & 1 & 1\\
						\hline
						2 & 2 & 2 & 2 & 2 & 2 \\
						\hline
						3 & 3 & 3 & 3 & 3 & 3 \\
						\hline
						4 & 4 & 4 & 4 & 4 & 4\\
						\hline
						5 & 5 & 5 & 5 & 5 & 5\\
						\hline
						6 & 6 & 6 & 6 & 6 & 6\\
						\hline
						7 & 7 & 7 & 7 & 7 & 7 \\
						\hline
						8 & 8 & 8 & 8 & 8 & 8\\
						\hline
						9 & 9 & 9 & 9 & 9 & 9\\
						\hline
						10 & 10 & 10 & 10 & 10 & 10\\
						\hline
					\end{tabular}
					\\
					\centering
					$Z_N^t$
				\end{minipage}	
				&  
				\begin{minipage}{.5\linewidth}
					\begin{tabular}{|p{5mm}|p{5mm}|p{5mm}|p{5mm}|p{5mm}|p{5mm}|}
						\hline
						1  & 1 & 1 & 1 & 1 & 1\\
						\hline
						1 &  1 & 1 & 1 & 2 & 2\\
						\hline
						2 &  2 & 2 & 2 & 2 & 2 \\
						\hline
						2 &  2 & 4 & 4 & 4 & 4 \\
						\hline
						4 &  4 & 4 & 4 & 4 & 4 \\
						\hline
						6 &  6 & 6 & 6 & 6 & 6\\
						\hline
						6 & 6  & 6  & 6 & 7 & 7\\
						\hline
						7 & 7  & 7  & 7 & 7 & 7 \\
						\hline
						7 & 7  & 9  & 9 & 9 & 9 \\
						\hline
						9 & 9  & 9  & 9 & 9 & 9 \\
						\hline
					\end{tabular}
					\\
					\centering
					$D_i$
				\end{minipage}	
			\end{tabular}
		\end{center}
		\caption{
			\label{tab:weakprop-implies-prop1-het}
			Example for the proof of \Cref{thm:weak-full-balance-implies-ordinal-prop1-het}.
			Here $n=10,t=6$, and the entries are picking rights.
		}
	\end{table}
	
	For illustration, \Cref{tab:weakprop-implies-prop1-het} shows the rights in $Z_N^t$ (left), and the rights $1$ and $\floor{j n / t + 1}$ for $j\in[t-1]$ in $D_i$ (right) for $n=10,t=6$. One can see that each right in the $D_i$ table is at least as good as the right at the same location in the $Z_N^t$ table, as required.
\end{proof}

\end{document}